\documentclass[11pt,reqno]{amsart}

\usepackage{graphicx}

\newcommand{\K}{{\mathbf K}^\crit}

\newcommand{\sm}{\setminus}

\newcommand{\kahler}{K\"ahler }

\newcommand{\wh}{\widehat}
\newcommand{\PP}{{\mathbb P}}
\newcommand{\RR}{\mathbb{R}}

\newcommand{\CC}{{\mathbb C}}

\newcommand{\ZZ}{{\mathbb Z}}
\newcommand{\NN}{{\mathbb N}}
\newcommand{\R}{{\mathbb R}}
\newcommand{\C}{{\mathbb C}}

\newcommand{\Z}{{\mathbb Z}}

\newcommand{\N}{{\mathbb N}}

\newcommand{\dbar}{\bar\partial}
\newcommand{\ddbar}{\partial\dbar}

\newcommand{\D}{{\mathbf D}}

\renewcommand{\H}{{\mathbf H}}

\newcommand{\hcal}{\mathcal{H}}

\newcommand{\lcal}{\mathcal{L}}

\newcommand{\pcal}{\mathcal{P}}

\newcommand{\qcal}{\mathcal{Q}}

\newcommand{\al}{\alpha}

\newcommand{\la}{\lambda}

\def    \span   {{\operatorname{span}}}

\def    \open   {{\operatorname{o}}}
\def    \span   {{\operatorname{span}}}

\def    \Z  {{\mathbb Z}}
\def    \R  {{\mathbb R}}
\def    \C  {{\mathbb C}}

 \def    \Im     {{\operatorname{Im}}}

\newtheorem{maintheo}{{\sc Theorem}}

\newtheorem{theo}{{\sc Theorem}}[section]
\newtheorem{cor}[theo]{{\sc Corollary}}
\newtheorem{conj}[theo]{{\sc Conjecture}}

\newtheorem{remark}[theo]{{\sc Remark}}
\newtheorem{lem}[theo]{{\sc Lemma}}
\newtheorem{prop}[theo]{{\sc Proposition}}
\newtheorem{definition}[theo]{{\sc Definition}}

\newenvironment{defin-no-number}{\medskip\noindent{\it Definition:\/} }{\medskip}

\newtheorem{claim}[theo]{{\sc Claim}}

\def\h#1{\hbox{#1}}
\def\HungarianAccent#1{{\accent"7D #1}}
\def\o{\omega}

\def\Szego{Szeg\HungarianAccent{o} }
\def\K{K\"ahler }

\def\ra{\rightarrow}
\def\a{\alpha}
\def\th{\theta}
\def\vp{\varphi}

\def\isom{\cong}

\def\w{\wedge}
\def\i{\sqrt{-1}}
\def\text{\textstyle}
\def\equationspace{\medskip\noindent}

\def\ra{\rightarrow}

\def\isom{\cong}

\def\del{\partial}

\def\MAop{\hbox{\rm MA\hglue0.02cm}}

\def\dis{\displaystyle}

\def\calQ{\qcal}
\def\calP{\pcal}

\def\calL{\lcal}

\def\calH{\hcal} \def\H{\hcal}

\newcommand{\Hilb}{{{\operatorname{Hilb}}}}

\def\uscreg{{\operatorname{reg}}}
\def\cvx{{\operatorname{cvx}}}

\def\sing{{\operatorname{sing}}}
\def\max{{\operatorname{max}}}

\title[
The Cauchy problem for Monge-Amp\`ere, I
]
{The Cauchy problem for the homogeneous Monge-Amp\`ere equation, I. Toeplitz quantization
}

\author{Yanir A. Rubinstein }
\author{Steve Zelditch }

\address{Department of Mathematics, Stanford University, Stanford, CA 94305, USA}
\email{yanir@member.ams.org}

\address{Department of Mathematics, Northwestern  University,
Evanston, IL 60208, USA} \email{ zelditch@math.northwestern.edu}

\thanks{\hglue-10pt August 20, 2010. Revised October 26, 2010.}

\begin{document}

\maketitle

\maketitle

\begin{abstract}

The Cauchy problem for the homogeneous  (real and complex)
Monge-Amp\`ere equation (HRMA/HCMA) arises from the initial value
problem  
for geo-desics in the space of K\"ahler metrics. It
is an ill-posed problem.  We conjecture that, in its  lifespan,
the solution can be obtained by  Toeplitz quantizing the
Hamiltonian flow defined by  the Cauchy data, analytically
continuing the quantization, and then taking a kind of logarithmic
classical limit. In this article, we prove that in the case of
torus invariant metrics 
(where the
HCMA reduces to the HRMA) this ``quantum analytic continuation
potential" coincides with the well-known Legendre transform
potential, and hence solves the equation as long as it is
smooth. In the sequel \cite{RZ2} we prove that the Legendre
transform potential ceases to solve the HRMA after that
time.

\end{abstract}

\bigskip
\section{Introduction}

This article is the first in a series whose aim is to study
existence, uniqueness and regularity of solutions of the initial
value problem (IVP) for geodesics in the space of \kahler metrics
in a fixed class. It is a special case of the  Cauchy problem for
the HCMA (homogeneous complex Monge-Amp\`ere equation).
Unlike the much-studied Dirichlet problem 
little has been proven for the Cauchy problem for the
Monge-Amp\`ere equation, and there is currently no known method to
solve it for smooth Cauchy data. Indeed, it is an ill-posed
problem and one does not expect global in time solutions to exist
for `most' initial data. The goal is thus to determine which
initial data give rise to global solutions, especially those of
relevance in geometry (`geodesic rays') and to determine the
lifespan $T_\span$ of solutions for general initial data. In this
article, we propose a general solution to the  IVP  for the
geodesic equation on a polarized projective \kahler manifold,
valid for the lifespan of the solution, in terms of a Toeplitz
quantization and its analytic continuation. This conjectural
solution,  which we call the ``quantum analytic continuation
potential," is defined as the logarithmic limit of   a  canonical
sequence of subsolutions of the HCMA  obtained from the analytic
continuation in time of the Toeplitz quantization of the Cauchy
data.

Our first goal in this series is to show that the conjectured solution
is indeed a
solution to the IVP for geodesics, as long as one exists, when the
\kahler manifold $(M, \omega)$ has an  $(S^1)^n$ symmetry with $n
= \dim M$. In such cases (including toric \kahler manifolds and
Abelian varieties), the HCMA reduces to the HRMA (homogeneous real
Monge-Amp\`ere equation). Even in this setting, the problem is
rather involved, and its different aspects are treated separately
in the different articles of the series.   In this article, we
prove that in the $(S^1)^n$-invariant case, the quantum analytic
continuation potential is a Lipschitz continuous subsolution that
is a smooth solution of the HRMA until the `convex lifespan'
$T_\span^\cvx$ of the problem (see Definition
\ref{ConvexLifeSpanDef}). In the sequel \cite{RZ2},  we show that
the quantum analytic continuation potential  fails to solve the
equation even in a weak sense after the convex lifespan. In
\cite{RZ3}, we characterize the smooth lifespan of the HCMA. 
In particular, for the HRMA, we show that the smooth lifespan
$T_\span^\infty$ (see Definition \ref{LifeSpanDef}) of the Cauchy
problem equals the convex lifespan. 
Hence the directions of
smooth geodesic rays are those with infinite convex lifespan.

This article and the next one \cite{RZ2} are devoted mainly to
the HRMA and to \K manifolds with symmetry.
However, the quantum analytic continuation potential constructed in this
article (see \S\ref{QuantumApproachSection} and \S\ref{QuantizingSection}), 
and the characterization of the 
smooth lifespan in \cite{RZ3}, 
apply to the HCMA and to general \K manifolds. 
In addition, we believe that the rest of the methods developed here have 
natural extensions at least to the case of Riemann surfaces.

Our study is to a large extent motivated by applications to \K geometry,
that we now briefly describe. Let $(M,J,\omega)$ denote a closed compact 
\K manifold of complex dimension $n$.
Consider the infinite-dimensional space
\begin{equation}
\label{HoEq}
\calH_\o
=
\{\vp\in C^{\infty}(M) \,:\, \omega_\vp:= \omega+\i\ddbar\vp>0\},
\end{equation}
of \K metrics in a fixed \K class equipped with the Riemannian metric \cite{M,S,D1}
\begin{equation}
\label{metric}
g_{L^2}(\zeta,\eta)_{\vp}:= \frac1V\int_M
\zeta\eta\, {\omega_{\vp}^m},\quad \vp \in \hcal_{\o},\quad
\zeta,\eta \in T_{\vp} \hcal_{\o}\isom C^\infty(M).
\end{equation}
One may show that covariant differentiation on $(\calH_\o,g_{L^2})$ is given by
\begin{equation}
\label{HConnectionEq}
D_c e=\dot e-{\textstyle\frac12} g_\vp(\nabla c,\nabla e),
\end{equation}
where $\gamma(s)$ is a curve in $\H_\o$ with $\gamma(0)=\vp,
\dot\gamma(0)=c\in T_\vp\H_\o$ and $e(s)=e(\gamma(s))$
is a vector field on $\H_\o$ along $\gamma$. Here $g_\vp$ is the Riemannian
metric associated to $\o_\vp$ and $\nabla$ is the Levi-Civita connection of $g_\vp$.
Hence, geodesics of $(\calH_\o,g_{L^2})$ are maps $\vp$ from a connected subset $I$ of $\RR$ to $\calH_\o$, equivalently
functions on $I\times M$,
that satisfy the equation
\begin{equation}
\label{HoGeodEq}
\ddot\vp-{\textstyle\frac12} g_\vp(\nabla \dot{\vp},\nabla \dot{\vp})=0,
\quad\hbox{on}\quad (I\setminus\partial I) \times M.
\end{equation}
Extend $\vp$ in a trivial manner to $(I\setminus\partial I)\times \RR \times M$,
i.e., by setting $\vp$ to be $\RR$-invariant, and denote
by $\pi_2$ the projection map from this product to $M$,
and by $\tau=s+\i t$ the holomorphic coordinate on
$(I\setminus\partial I)\times \RR$.
It was observed by Semmes and Donaldson that
\begin{equation}
\label{HCMAEq}
\begin{array}{lll}
&\displaystyle
\frac1{n+1}(\pi_2^\star\omega + \i\ddbar \vp)^{n+1}=
\cr\cr
& \displaystyle
\big(\ddot\vp-{\textstyle\frac12} g_\vp(\nabla \dot{\vp},\nabla \dot{\vp})\big)
\i d\tau\w d\bar\tau\w\o_\vp^n,
\quad\hbox{on}\quad (I\setminus\partial I)\times \RR \times M.
\end{array}
\end{equation}
Therefore, when $\vp$ is regular enough, the geodesic equation is
equivalent to the homogeneous complex Monge-Amp\`ere (HCMA) equation
on the product of a Riemann surface with $M$.

The initial value problem is the problem of defining the
exponential map of $\hcal_{\omega}$. 
Although the Cauchy problem
is ill-posed for the HCMA, infinite geodesic rays are expected to
play an important role in \K geometry and 
this is one motivation to study the IVP
(see \cite{AT,Ch,CTa,CT,D1,M,PS2,PS3,S,Su} for
relevant \K geometry background). Yet the ill-posedness makes
the Cauchy problem very different from the Dirichlet problem
corresponding to geodesics connecting two given end-points,
whose existence and regularity was first studied extensively by 
Chen \cite{Ch}, 
Donaldson \cite{D2}, and Chen-Tian \cite{CT}.
As observed by Mabuchi, Semmes, and Donaldson,  $\hcal_{\omega}$ is
formally an infinite dimensional symmetric space of the type $G^{\C}/G$
where $G $ is the group of Hamiltonian diffeomorphisms of $(M, \omega)$.
Hence its geodesics should be given by certain one-parameter subgroups
of $G^{\C}$, which correspond to analytic continuations in time of
Hamiltonian orbits. To a large extent, the \K quantization
method of this article is an attempt to put these formal arguments
on a rigorous basis.

The article is organized as follows. In Section \ref{QuantumApproachSection}
we describe our approach to the IVP using an analytic continuation
of Toeplitz quantization.
Our main results are stated in Section \ref{StatementResultsSection},
and in Section \ref{SectionBackground} we recall some background.
In Section \ref{QuantizingSection} we construct the quantization
of the Hamiltonian flow. The results in this Section hold on
an arbitrary projective \K manifold. In Section
\ref{TwoQuantizationsSection} we specialize to the setting
of a toric or Abelian variety where we construct a second
quantization of the Hamiltonian flow and compare the two
quantizations and their analytic continuations. In Section \ref{ConvergenceSection}
we complete the proof of our  main result (Theorem \ref{FirstMainThm}),
showing that the analytic continuations of the quantizations converge
to the Legendre transform potential and solve the Cauchy problem
until the convex lifespan.

\bigskip

\section{A Quantum mechanical approach to Monge-Amp\`ere}
\label{QuantumApproachSection}

\bigskip

In this section we define the {\it quantum analytic continuation
potential}  and state the general conjecture that it solves the
IVP for geodesics in $(\calH_\o,g_{L^2})$, to the extent possible,
in the case of projective \K manifolds. The definition  is
inspired by two prior  constructions and is largely aimed at
reconciling them.

The first is a heuristic
analytic continuation argument due to Semmes and Donaldson \cite{S,D}: Let
$\dot\vp_0$ be a smooth function on $M$, considered as a tangent
vector in $T_{\vp_0}\calH_\o$. Let
$X^{\o_{\vp_0}}_{\dot\vp_0}\equiv X_{\dot\vp_0}$ denote the
Hamiltonian vector field associated to $\dot\vp_0$ and
$(M,\omega_{\vp_0})$ and let $\exp t X_{\dot\vp_0}$ denote the
associated Hamiltonian flow. Then let $\exp \i s X_{\dot{\vp}_0}$
``be" its analytic continuation in time to the Hamiltonian flow at
``imaginary" time $\i s$. Then ``define" the {\it classical
analytic continuation potential} $\vp_s$ with initial data
$(\vp_0, \dot{\vp}_0)$  by
\begin{equation}
\label{SOL}
(\exp \i s X_{\dot{\vp}_0})^\star \omega_0 - \omega_0 = \i\ddbar \vp_s.
\end{equation}
Then $\vp_s$ ``is" the solution of the initial value problem.
We use quotes since there is no obvious reason why
$\exp t X_{\dot{\vp}_0}$, a rather arbitrary smooth
Hamiltonian flow, should admit an analytic continuation in $t$ for any
length of time. When the analytic continuation does exist, e.g., if
$\omega_{\vp_0}$ and $\dot{\vp}_0$ are real analytic, then $\vp_s$ solves
the initial value problem for the Monge-Amp\`ere equation for $s$ in some (usually) small time interval.

The second construction uses finite dimensional approximations
deriving from \K quantization.  The idea is to approximate the
space $\calH_\o$ by finite-dimensional spaces of Bergman (or
Fubini-Study)  metrics induced by holomorphic embeddings of  $M$
into $\PP^N$ using bases of holomorphic sections $s \in H^0(M,
L^k)$ of high powers of a polarizing line bundle. Following an
original idea of Yau and Tian, such embeddings were used in
\cite{T,C,Z3} to approximate individual metrics. Phong-Sturm
\cite{PS1,PS2} then introduced a \kahler quantization method to
approximate geodesic segments with fixed end-points by geodesics in
the space of Bergman metrics. They also used the method to define
geodesic rays from test configurations. Further work on Bergman
approximations to geodesics, as well as more general harmonic
maps, are due to Berndtsson, Chen-Sun, Feng, Song-Zelditch, and
others \cite{B1,B2,CS, Fe, RZ1,SoZ1,SoZ2}.

Our approach combines the two as follows: we define the analytic
continuation of $\exp t X_{\dot{\vp}_0}$ by quantizing this
Hamiltonian flow, by analytically continuing the quantum flow, and
then by taking a kind of logarithmic classical limit of its
Schwartz kernel.

Consider the Hilbert spaces of sections $L^2(M, L^N), N\in\NN$,
associated to powers of a Hermitian line bundle $(L,h_0)$ polarizing $(M,\o_{\vp_0})$,
and the corresponding orthogonal projection operators
$$
\Pi_N\equiv\Pi_{N,\vp_0}:L^2(M, L^N) \to H^0(M, L^N),
$$
onto the Hilbert subspaces $H^0(M,L^N)$ of holomorphic sections.
These Hilbert subspaces allow one to `quantize' $(M,\omega_{\vp_0})$. In order
to quantize the Hamiltonian flow of $X_{\dot\vp_0}$ on $(M,\omega_{\vp_0})$ we use
the method of Toeplitz quantization. Namely, we consider the operators
$$
\Pi_{N}\circ\dot\vp_0\circ\Pi_{N},
$$
where here $\dot\vp_0$ denotes the operator of multiplication by $\dot\vp_0$.
We
will usually omit the composition symbols and denote these by
$\Pi_{N}\dot\vp_0\Pi_{N}$.
These are zero-order self-adjoint operators.
Define the associated one-parameter subgroups of unitary operators

\begin{equation} \label{UNDEF}
U_N(t):=\Pi_N e^{\i tN \Pi_N\dot\vp_0\Pi_N}\Pi_N \end{equation}
\smallskip
\noindent
on $H^0(M,L^N)$.

A key observation is that there is no obstruction to analytically
continuing the quantization: each $U_N(t)$ admits an analytic
continuation in time $t$ and induces the imaginary time semi-group

\begin{equation}
\label{UKDEF}
U_N(\i s)  : H^0(M, L^N) \to H^0(M, L^N), \quad U_N(\i s)\in GL(H^0(M, L^N),\CC).
\end{equation}

\smallskip
\noindent The main idea of this article is that  the analytic
continuation of $\exp t X_{\dot\vp_0}$ can be constructed by
taking a non-standard kind of logarithmic  classical limit of the
analytic continuation of its quantization. We do this by
considering the Schwartz kernel $U_N(-\i s)(z,w)$ of this operator
with respect to the volume form $(N\omega_{\vp_0})^n$.

\begin{definition}
\label{QuantumAnalyticPotentialDef}
Set
\begin{equation}
\label{NAnalyticContPotential}
\vp_N(s,z):= \frac{1}{N} \log U_N(-\i s,z,z).
\end{equation}
We define the {\it quantum  analytic continuation potential} $\vp_\infty$ by
$$
\vp_\infty(s,z)
:=
\lim_{l \to \infty} (\sup_{N \geq l}\vp_N)_\uscreg(s,z).
$$
\end{definition}

Here,
$u_\uscreg(z_0):= \lim_{\epsilon \to 0} \sup_{|z - z_0| < \epsilon} u(z)$
denotes the upper semi-continuous regularization of $u$.
The limit on the right hand
side exists and is $\pi_2^\star\o$-plurisubharmonic,
since it is a limit of a sequence of decreasing
$\pi_2^\star\o$-psh functions (\cite{De2}, \S I.5).

This limit is quite different from the semi-classical limits
studied in Toeplitz quantization, because the analytic
continuation in time destroys the Toeplitz structure of the
kernel. Moreover, the  logarithmic asymptotics of the Schwartz
kernel is quite unrelated to  symbol asymptotics. One may think of
it as extracting an analytic continuation of the `phase function'
of the Toeplitz operator; the `symbol' of the Toeplitz operator is
irrelevant.

Denote by
$$
S_T:=[0,T]\times\RR
$$

\smallskip
\noindent
the (vertical) strip of width $T$ in $\CC$.
The IVP for geodesics is equivalent to the following
Cauchy problem for the homogeneous complex Monge-Amp\`ere equation:

\begin{equation}
\label{HCMARayEq}
\left\{
\begin{array}{rrl}
(\pi_2^\star\omega + \i\ddbar \vp)^{n+1}
\!\!\!& = & \!\!\!
0 \quad\quad\;\,\mskip2mu \mbox{on} \; S_{T} \times M,
\cr\cr
\vp(0,s,\,\cdot\,)
\!\!\!& = &\!\!\!
\vp_0(\,\cdot\,)
\;\; \mbox{on} \; \{0\}\times\RR \times M,
\cr\cr
\displaystyle
\frac{\partial\vp}{\partial s}(0,s,\,\cdot\,)
\!\!\!& = &\!\!\!
\dot\vp_0(\,\cdot\,) \;\; \mbox{on} \; \{0\}\times\RR \times M,
\end{array} \right.
\end{equation}

\bigskip
\noindent
Note here that the complex Monge-Amp\`ere operator is well-defined on bounded plurisubharmonic functions
\cite{BT1,BT2}.

\begin{definition}
\label{LifeSpanDef}
We define the  smooth lifespan (respectively, lifespan)
of the Cauchy problem (\ref{HCMARayEq}) 
to be the supremum over all $T\ge 0$ such that
(\ref{HCMARayEq}) admits a smooth (respectively $\pi^\star_2\omega$-psh) solution.
We denote the smooth lifespan (respectively, lifespan) for the Cauchy data
$(\o_{\vp_0},\dot\vp_0)$
by  $T^\infty_\span\equiv T^\infty_\span(\o_{\vp_0},\dot\vp_0)$ (respectively,
$T_\span\equiv T_\span(\o_{\vp_0},\dot\vp_0)$).
\end{definition}

\begin{definition}
Define the quantum lifespan $T_\span^Q$ 
of the Cauchy problem (\ref{HCMARayEq}) 
to be supremum over all $T\ge0$ such that
the {\it quantum  analytic continuation potential} $\vp_\infty$ solves the
HCMA (\ref{HCMARayEq}).
\end{definition}

 We pose the following
conjecture, which would give a general method to solve the
ill-posed Cauchy problem for the HCMA to the extent possible.

\begin{conj}
\label{FirstMainConj}

The quantum analytic continuation potential $\vp_\infty$ solves the HCMA (\ref{HCMARayEq})
for as long as it admits a solution. In other words, $T_\span^Q=T_\span$.

\end{conj}

As mentioned above, the key difficulty in the analysis is that
although  $U_N(t, z, w)$ is a standard Toeplitz Fourier integral
operator  quantizing the Hamilton flow of $\dot\vp_0$,   its
analytic continuation $U_N(-\i s, z,z)$ lies outside the class  of
complex Fourier integral operators, and it is difficult to analyze
its logarithmic asymptotics or  to determine how  regular the
limit should be.  The toric setting provides a testing ground
where it is possible to make a complete analysis. We only give the
details for toric \kahler manifolds, but as in \cite{Fe}, the same
methods apply to Abelian varieties.

\bigskip
\section{Statement of results}
\label{StatementResultsSection}
\bigskip

The main results of this article concern the Cauchy problem for
the HRMA. While the Dirichlet problem for the HRMA has been
extensively studied
(see \cite{RT,
CNS,
GTW,Gz} and references therein), the Cauchy problem has not been
systematically  investigated. We are only aware of
\cite{BB} that proves uniqueness of $C^3$ solutions for the
Cauchy problem for the more general HCMA,  of \cite{Fo1,Fo2},
where a sufficient condition on the Cauchy
data is given for existence of a smooth short-time solution of HRMA depending
on the Cauchy hypersurface
(for our Cauchy hypersurface, the existence of a smooth short-time solution
is not an issue, since it follows independently from a classical Legendre duality argument),
and of \cite{U} where an explicit formula is derived for smooth solutions
of the 2-dimensional HRMA.

In general, the HRMA can be viewed as a special case of the HCMA under the presence
of sufficient symmetry. In the setting of the HCMA (\ref{HCMAEq}) corresponding
to the IVP for geodesics, the reduction to a HRMA precisely corresponds
to restricting from a general projective variety to a toric or Abelian one.
Let us now describe briefly this geometric setting, concentrating on
the toric case (for more background see \S\S\ref{ToricBackgroundSubsection}).


\def\bfT{{\bf T}}

A toric \K manifold is a \K manifold $(M,J,\o)$ that admits a holomorphic action of a complex
torus $(\CC^\star)^n$ with an open dense orbit, and for which the \K form $\o$ is toric, i.e.,
invariant under the action of the real torus
$$
\bfT:=(S^1)^n.
$$
We assume that the Cauchy data $(\omega_{\vp_0},\dot\vp_0)$ is toric,
and consider the IVP for geodesics in the space of torus-invariant \K metrics.
Over the open orbit
$$
M_\open\isom(\CC^\star)^n\isom \RR^n\times \bfT
$$
the \K form $\omega_{\vp_0}$ is exact and $\bfT$-invariant and so
we let $\psi_0$ be a smooth strictly convex function on $\RR^n$
satisfying
\begin{equation}
\label{PsiZeroDefEq}
\o_{\vp_0}|_{M_\open}=\i\ddbar\psi_0.
\end{equation}
Here $[\o]$ is any integral \K class in $H^2(M,\ZZ)$. The initial
velocity $\dot\vp_0$ is also $\bfT$- invariant, and so it induces,
by restriction to the open orbit, a smooth bounded function on $\RR^n$,
that we denote by $\dot\psi_0$. Analytically, the
IVP is then equivalent to studying the following HRMA for a convex
function $\psi$ on $[0,T]\times\RR^n$,

\begin{equation}
\label{HRMARayEq}
\left\{
\begin{array}{rrl}
\h{\rm MA}\, \psi
\!\!\!& = & \!\!\!
0, \quad\quad\;\,\mskip2mu \mbox{on} \; [0,T] \times \RR^n,
\cr\cr
\psi(0,\,\cdot\,)
\!\!\!& = &\!\!\!
\psi_0(\,\cdot\,),
\;\; \mbox{on} \;  \RR^n,
\cr\cr
\displaystyle
\frac{\partial\psi}{\partial s}(0,\,\cdot\,)
\!\!\!& = &\!\!\!
\dot\psi_0(\,\cdot\,), \;\; \mbox{on} \; \RR^n.
\end{array} \right.
\end{equation}
\smallskip

Here, $\h{\rm MA}$ denotes the real Monge-Amp\`ere operator that can be
defined as a Borel measure on convex functions

$$
\h{\rm MA}\, f:=
d\frac{\partial f}{\partial x^1}\w \cdots\w d\frac{\partial f}{\partial x^{n+1}},
\quad \h{for } f \h{\ convex on } \RR^{n+1},
$$

\smallskip
\noindent and equals $\det\nabla^2 f\, dx^1\w\cdots\w dx^{n+1}$ on
$C^2$ functions \cite{RT}.


Let 
$$
P:=\overline{\Im\nabla\psi_0}\subset\RR^n.
$$ 
Recall that on a
symplectic toric manifold the Legendre transform $f\mapsto
f^\star$ is a bijection between the set of $\bfT$-invariant \K
potentials on the open orbit $M_\open\isom (\CC^n)^\star$ of the
(complex) torus action

$$
\calH(\bfT):=\{\psi\in C^\infty(\RR^n)\,:\, \i\ddbar\psi=\o_\vp|_{M_\open} \hbox{\ with\ } \vp\in\calH_\o
\h{\ and\ } \overline{\Im\nabla\psi}=P\},
$$

\smallskip
\noindent
and the set of
symplectic potentials on the moment polytope $P\subset \RR^n$

\begin{equation}
\label{LHTDefEq}
\calL\calH(\bfT):=\{u\in C^\infty(P\setminus\partial P)\cap C^0(P)\,:\, u=\psi^\star
\hbox{\ with\ } \psi\in\calH(\bfT)\}.
\end{equation}

\smallskip
\noindent
When the latter space is equipped with the standard $L^2(P)$ metric, this map is in fact an isometry
and transforms the IVP geodesic equation (\ref{HoGeodEq}) to the linear equation
\begin{equation}
\label{IVPToricGeodEq}
\ddot u=0, \quad u_0=\psi^\star_0, \quad \dot u_0=-\dot\psi_0\circ(\nabla \psi_0)^{-1},
\end{equation}
whose solution is given by
$$
u_s:=u_0+s\dot u_0.
$$

\begin{definition}
\label
{ConvexLifeSpanDef}
Define the convex lifespan of the Cauchy problem (\ref{HRMARayEq}) as
$$
T_\span^\cvx(\psi_0,\dot\psi_0):=
\,\sup\,\{\,s\,:\, \psi_0^\star-s\dot\psi_0\circ(\nabla\psi_0)^{-1} \h{\ is convex on $P$}\,\}.
$$
\end{definition}

\smallskip
\noindent
We note that $T_\span^\cvx$ is independent of the choice of $\psi_0$
satisfying (\ref{PsiZeroDefEq}).

At least as long as $s<T_\span^\cvx$, i.e., $u_s$ is strictly convex
and hence belongs to $\calL\calH(\bfT)$,
it is well-known that the IVP for geodesics
has an explicit solution,

\begin{equation}
\label{IVPToricGeodFormula}
\psi(s,x)=\psi_s(x):= (u_0 + s\dot{u}_0)^\star(x),\quad s\in[0,T_\span),\; x\in\RR^n.
\end{equation}

\smallskip
\noindent
For a review of this fact and references
we refer to \cite{RZ2}.
We call $\psi$ the {\it Legendre transform potential}.

What is less transparent is what happens when $s>T_\span^\cvx$.
Firstly, it should be pointed out that, as defined in
(\ref{IVPToricGeodFormula}), $\psi_s$ is  
finite for each $x\in\RR^n$. Hence, it is necessarily
Lipschitz. Moreover, as we show in \cite{RZ2},
$\psi_s$ is strictly convex, but not differetiable
everywhere.


Denote by $\H^{0,1}(\bfT)$ the closure of $\calH(\bfT)$
with respect to the $C^{0,1}$-norm (this space contains also convex functions
that are not strictly convex).
The corresponding space of
$\o$-psh (plurisubharmonic) functions will be denoted by $\H_\o^{0,1}$. 
According to the previous paragraph,
one has $\psi_s\in\H^{0,1}(\bfT)$ for all $s>0$.
It therefore makes sense to consider $\psi$ as an infinite ray
in the interior of $\calH^{0,1}(\bfT)$.

%
%

Our main result in this article states that the sequence of level $N$ quantum
analytic continuation potentials
$\vp_N$ defined by (\ref{NAnalyticContPotential})
converges uniformly to the Legendre transform potential $\psi$, and
therefore the quantum analytic continuation potential $\vp_\infty$
of Definition \ref{QuantumAnalyticPotentialDef}
solves the HCMA for $T<T_\span^\cvx$.

\begin{maintheo}
\label{FirstMainThm}
Let $\vp:=\psi-\psi_0$ be the one-parameter family of
Lipschitz continuous $\omega$-psh potentials associated to the Legendre transform
potential $\psi$ given by (\ref{IVPToricGeodFormula}),
and let $\vp_N$ be the quantum
analytic continuation potentials given by (\ref{NAnalyticContPotential}).
Then
$$
\lim_{N\ra\infty}\vp_N = \vp
$$
in $C^2([0, T] \times M)$ for $T<T^\cvx_{\span}$, and in $C^0([0, T]
\times M)$ for $T\ge T^\cvx_{\span}$.
In particular, the quantum analytic continuation potential coincides with the
Legendre transform potential
$$
\vp_\infty=\vp\in \H_\o^{0,1}.
$$
\end{maintheo}

\medskip

In the sequel, we prove that the quantum analytic continuation
potential $\vp$ ceases to solve the HCMA (\ref{HCMARayEq}) for any
$T>T^\cvx_\span$. Moreover, we show that on a dense set,
whose complement has zero Lebesgue measure, it
does solve the equation. We state the result in terms of the
failure to solve the corresponding HRMA (\ref{HRMARayEq}), that
corresponds to the HCMA on the open orbit $M_\open$. Let
$$
\Delta(\psi):= \{\,(s,x)\,:\, \psi \h{\rm\  is finite and differentiable
at $(s,x)$\,}\}\subset \RR_+\times\RR^n,
$$ 
denote the regular locus of $\psi$, and let
$$
\Sigma_\sing:=\,\RR_+\times\RR^n\;\sm\, \Delta(\psi),
$$
denote its singular locus. Since $\psi$
is everywhere finite, the former is dense while 
the latter has Lebesgue measure zero in
$\RR_+\times\RR^n$.
Set,
$$
\Sigma_\sing(T):=\,[0,T]\times\RR^n\;\sm\, \Delta(\psi).
$$

\begin{maintheo}
\label{SecondMainThm} {\rm (See \cite{RZ2}.)} 
(i)\; $\psi$ solves the HRMA (\ref{HRMARayEq}) on
the dense regular locus,
$$\MAop\psi=0\quad \h{\ on\ }\quad \Delta(\psi)\subset\RR_+\times
\RR^n.
$$In addition, $[0,T_\span^\cvx)\times\RR^n\subset
\Delta(\psi)$.\hfill\break
(ii)\; Whenever $T>T_\span^\cvx$, $\psi$ fails to solve
the HRMA (\ref{HRMARayEq}). In particular, the Monge-Amp\`ere
measure of $\psi$ charges the set $\Sigma_\sing(T)$ with positive
mass, 
$$
\int_{[0,T]\times\RR^n} \MAop \psi =\int_{\Sigma_\sing(T)} \MAop
\psi>0.
$$
Equivalently, $\vp=\vp_\infty$ ceases to solve the HCMA
(\ref{HCMARayEq}), when $T>T_\span^\cvx$. However, it does solve
the HCMA on a dense set in $S_T\times M$.

\end{maintheo}

It is  well-known that the Legendre transform linearizes the HRMA,
and hence that the Legendre transform potential $\psi$ is a solution as
long as it is sufficiently smooth or equivalently as long as the
symplectic potential is strictly convex. It does not seem to have been
observed before that the Legendre transform potential fails to solve the
HRMA as soon as it ceases to be differentiable.
Theorems \ref{FirstMainThm} and \ref{SecondMainThm} come close to
settling Conjecture \ref{FirstMainConj} in the case of toric or
Abelian varieties. They leave open the possibility that there
exists an alternative method to solve the HRMA. That possibility is
investigated in \cite{RZ3}, where it is shown that the Legendre
solution is in a sense the optimal subsolution among several
natural approaches.

In order to prove Theorem \ref{FirstMainThm} we first show that
the operators $U_N$ quantize the Hamiltonian flow of
$X_{\dot\vp_0}^{\o_{\vp_0}}$. This result holds on any projective
\K manifold and does not make use of symmetry. The proof is based
on   the Toeplitz calculus developed by Boutet de
Monvel-Sj\"ostrand \cite{BSj} and Boutet de Monvel-Guillemin
\cite{BG}. We then show that $U_N$ is well approximated by a
second type of quantization that uses the symplectic potential.

The analysis of the logarithmic asymptotics of $U_N(-\i s, z,z)$
is closely related to the analysis of families of toric Bergman metrics
in \cite{SoZ1,Z4}, and these techniques allow us to compute the asymptotic
spectrum of these operators and conclude the $C^2$ convergence up
to $T<T_\span^\cvx$. Finally, we prove the global $C^0$
convergence to the Legendre transform subsolution.
The logarithmic
classical limit  is closer to large deviations theory than to
semi-classical Toeplitz analysis  since it involves the analytic
continuation in time of the Toeplitz quantization and not the
quantization itself.

\subsection{Further results}

As mentioned above, we prove in  \cite{RZ2} that the Legendre
transform potential fails to solve the equation even in a weak
sense after the convex lifespan. Consequently the quantization
method fails to solve the equation after this time, at least
in the case of the HRMA.

But it is plausible that the quantization method produces the
solution as long as a weak solution exists, and that it is in some
sense the ``optimal" sub-solution. To prove this, it is necessary to
investigate whether there  exist other ways of solving the Cauchy
problem after the convex lifespan. This is initiated in a
subsequent article \cite{RZ3} in the series where we 
characterize the smooth lifespan of the more general HCMA
in terms of analytic continuation of Hamiltonian dynamics.
In the case of the HRMA this characterization shows
precisely that $T_\span^\infty=T_\span^\cvx$, and 
hence no smooth solution exist beyond the convex lifespan.
By Theorem \ref{FirstMainThm} this shows that the 
quantization approach solves the Cauchy problem 
for as long as a smooth solution exists.

We also introduce  the notion of
a leafwise subsolution, and show that the Legendre transform
potential is the unique leafwise subsolution to the Cauchy
problem. Also, in a further sequel we show that in a certain class of admissible
subsolutions it is impossible to solve the Cauchy problem for the
HRMA beyond the convex lifespan. 
This comes sufficiently  close to confirming Conjecture
\ref{FirstMainConj} in the cases of toric \kahler manifolds and
Abelian varieties with Cauchy data invariant under $(S^1)^n$.

Among \kahler manifolds without large symmetry, it seems most
feasible to study the Cauchy problem for HCMA on a Riemann
surface. The results and methods of this series suggest a general
conjecture on the lifespan of solutions in that case. We plan to
discuss it elsewhere.

\bigskip
\section{
Background
}
\label{SectionBackground}
\bigskip


\subsection{K\"ahler quantization}
\label
{KahlerQuantSubsection}
\medskip

Our setting consists of a \K manifold $(M, \omega)$ of complex
dimension $n$ with $[\omega] \in H^2(M, \Z)$. Under this
integrality condition, there exists a positive Hermitian
holomorphic line bundle $(L, h) \to M$ whose  curvature form is
given locally by
$$
\o\equiv\o_h=-\frac{\sqrt{-1}}{2\pi}\ddbar \log \|e_L\|_h\;,
$$ where $e_L$ is a
nonvanishing local holomorphic section of $L$, and where
$\|e_L\|_h=h(e_L,e_L)^{1/2}$ denotes the $h$-norm of $e_L$.

The Hilbert spaces `quantizing' $(M, \omega)$  are then defined to
be the  spaces
$$
H^0(M,L^N)
$$
of holomorphic sections of
$L^N=L\otimes\cdots\otimes L$.  The metric $h$ induces Hermitian
metrics $h^N$ on $L^N$ given by $\|s^{\otimes
N}\|_{h_N}=\|s\|_h^N$. We give $L^2(M,L^N)$ the inner product
\begin{equation}
\label{HilbDef} ||s||^2_{\Hilb_N(h)}:=  \frac1V\int_M |s|_{h^N}^2
(N\o_h)^n.
\end{equation}
We then define the \Szego kernels as the Schwartz kernels
$\Pi_N(z,w)$ of the orthogonal projections $\Pi_N:L^2(M, L^N) \to
H^0(M, L^N)$ with respect to this inner product, so that
\begin{equation}
(\Pi_N s)(y)=\int_M \Pi_N(x,y)s(x)(N\o(x))^n, \quad s\in L^2(M,
L^N).
\end{equation}
(Note that $\Pi_N$ depends on $h$ although we omit that from the
notation.)

Instead of dealing with sequences of Hilbert spaces, observables
and unitary operators on $M$, it is convenient to lift them to the
circle bundle
$$
X=\{\la \in L^\star : \|\la\|_{h^{-1}}= 1\},
$$
where $L^\star$ is the dual line bundle to $L$, and  where $h^{-1}$ is
the norm on $L^\star$ dual to $h$. Let us now describe the lifted
objects.

Let $\rho$ be the function $||\lambda||_{h^{-1}}-1$ on $L^\star$.
Associated to $X$ is the contact form
$\al=-\i\partial\rho|_X=\i\dbar\rho|_X$ and  the volume form
\begin{equation}
\label{dvx} (d\al)^n\wedge\al=\pi^\star\o^n\wedge \al.
\end{equation}
We let $r_{\theta}w=e^{\i\theta} w,\; w\in X$, denote the $S^1$
action on $X$ and denote its infinitesimal generator by
$\frac{1}{\i}\frac{\partial}{\partial\theta}$. Holomorphic sections then lift
to elements of the  Hardy space $H^2(X) \subset L^2(X)$ of
square-integrable CR functions on $X$, i.e., functions that are
annihilated by the Cauchy-Riemann operator
$\dbar_b:=\pi^{0,1}\circ d$ (where $TX\otimes_{\RR}\CC=
T^{1,0}X\oplus T^{0,1}X\oplus\CC\frac{\partial}{\partial\theta}$
and $\pi^{0,1}$ is defined as the projection onto the second
factor) and are $L^2$ with respect to the inner product
\begin{equation}
\label{unitary}
\langle  F_1, F_2\rangle
=
\frac{1}{2\pi V}\int_X
F_1\overline{F_2}\;(d\al)^n\wedge\al ,\quad F_1,F_2\in L^2(X).
\end{equation}

The $S^1$ action on $X$ gives a representation of $S^1$ on
$L^2(X)$ with irreducible pieces denoted $L^2_N(X)$. We thus have
the Fourier decomposition,
\begin{equation}
\label{LtwoXBlockDecompositionEq} 
L^2(X)
=
\bigoplus_{N\ge0} L^2_N(X).
\end{equation}
We denote by $\D$ the operator on $L^2(X)$ with spectrum $\ZZ$ and
whose $N$-th eigenspace $L^2_N(X)$ consists of functions
transforming by $e^{\i N\theta}$ under the $S^1$ action
$r_{\theta}$ on $X$. Thus, 
\begin{equation} 
\label{DDEF} 
\D 
=
\frac{1}{\i} \frac{\partial}{\partial \theta}, \end{equation} 
the
infinitesimal generator of the $S^1$ action.

Since the $S^1$ action on $X$ commutes with $\bar{\partial}_b$ we
also have $H^2(X) = \bigoplus_{N =0}^{\infty}H^2_N(X)$ where
$$
H^2_N(X):= \{ F \in H^2(X): F(r_{\theta}w) = e^{\i N \theta} F(w)
\}=L^2_N(X)\cap \ker\bar{\partial}_b.
$$
A section $s_N$ of $L^N$ determines an equivariant function
$\hat{s}_N$ on $L^\star$ by the rule
\begin{equation}
\label
{LiftRuleEq}
\hat{s}_N(\lambda)
= \left( \lambda^{\otimes N}, s_N(z)\right), \quad \la\in
L^\star_z,\ z\in M,
\end{equation}
where $\lambda^{\otimes N} = \lambda \otimes \cdots\otimes
\lambda$. We henceforth restrict $\hat{s}$ to $X$ and then the
equivariance property takes the form $\hat s_N(r_\theta w) =
e^{iN\theta} \hat s_N(w)$. Up to a factor of $N^n$ the map
$s\mapsto \hat{s}$ is a unitary equivalence between $H^0(M, L^{
N})$ and $H^2_N(X)$.

We now define the (lifted) \Szego kernel of degree $N$ to be the
Schwartz kernel of the orthogonal projection $\tilde\Pi_N :
L^2(X)\rightarrow H^2_N(X)$. It is defined by
\begin{equation}
\label{PiNF}
\tilde\Pi_N F(w)
=
\frac{1}{2\pi V}\int_X
\tilde\Pi_N(w,v) F(v)\;(d\a)^n\wedge\a\,(v), \quad F\in L^2(X).
\end{equation}
The full \Szego kernel is then
\begin{equation}
\tilde\Pi = \sum_{N=1}^{\infty} \tilde\Pi_N.
\end{equation}

To simplify notation we will from now on omit the tilde from the
lifted projection operators on $X$ and simply write $\Pi,\Pi_N$.

It was proved by Boutet de Monvel and Sjostrand \cite{BSj} (see
also the Appendix to \cite{BG})  that $\Pi$ is a complex Fourier
integral operator (FIO) of positive type,
\begin{equation}
\label{FIO}
\Pi \in I_c^0(X \times X, {\mathcal C})
\end{equation}
associated to a positive canonical relation ${\mathcal C}$. For
definitions and notation concerning complex FIO we refer to
\cite{MS,BSj,BG}. The real points of ${\mathcal C}$ form the
diagonal $\Delta_{\Sigma \times \Sigma}$ in the square  of the
symplectic cone
\begin{equation}
\label{SympCone}
\Sigma
:=
\big\{\big(w,r \alpha(w)\big)\,:\, r > 0,\, w \in X\big\}
\subset
T^\star X,
\end{equation}
where $\alpha$ is the connection, or contact, form. We refer to \cite{BG},
Appendix, Lemma 4.5. Let $\omega_{T^\star X}$ denote the
canonical symplectic form on $T^\star X$, and let
\begin{equation}
\label{SympConeTwo}
\omega_\Sigma:=\omega_{T^\star X}|_\Sigma
\end{equation}
denote its restriction to $\Sigma$, a symplectic form on $\Sigma$.

Finally, recall that a Toeplitz operator is an operator of the
form $\Pi A\Pi$ where $A$ is a pseudo-differential operator, and a
(complex) Toeplitz Fourier integral operator is one where $A$ is
allowed to be a (complex) Fourier integral operator. 
When $A$ is a pseudo-differential operator we denote by $s_A$
its full symbol, and by $\sigma_A$ its principal symbol.
If $B$ is a (complex) Fourier integral operator we denote by
$\sigma_B$ its symbol.
Lastly, the symbol of $\Pi B \Pi$ is given by $\sigma_A|_\Sigma$ \cite{BSj}.

\subsection{Toric \kahler manifolds}
\label{ToricBackgroundSubsection}

We now review some geometry and analysis on toric \kahler
manifolds. Fuller details and exposition can be found in
\cite{A,G,R,RZ1,SoZ1,STZ}.

 Let $\bfT:=(S^1)^n$. A symplectic toric manifold is
a compact closed \K manifold $(M,\o)$ whose automorphism group
contains a complex torus $(\CC^\star)^n$ whose action on a generic
point is an open dense orbit isomorphic to $(\CC^\star)^n$, and for which the real
torus $\bfT\subset (\CC^\star)^n$ acts in a Hamiltonian fashion by
isometries.

We will work with coordinates on the open dense orbit
$$
M_\open\isom (\CC^\star)^n
$$
of the complex
torus given by
\begin{equation}
\label{CoordsOpenOrbitEq}
z=e^{x/2+\i\th}, \quad (x,\th)\in\R^n\times (S^1)^n.
\end{equation}
Let $\o|_{M_\open}=\i\ddbar\psi$.
The work of Atiyah and Guillemin-Sternberg \cite{At,GS2} implies
that the image of the moment map $\nabla\psi$ is a convex
polytope $P\subset\RR^n$ and depends
only on $[\o]$.
We further assume that this is a lattice polytope.
Being a lattice Delzant polytope \cite{De1} means that:
(i) at each vertex meet exactly $n$ edges,
(ii) each edge is the set of points $\{p+tu_{p,j}\,:\, t\ge0\}$ with
$p\in\ZZ^n$ a vertex, $u_{p,j}\in\ZZ^n$ and $\h{span}\{u_{p,1},\ldots,u_{p,n}\}=\ZZ^n$.
Equivalently, there exist outward pointing normal
vectors $\{v_j\}_{j=1}^d\subset\ZZ^n$, with $v_j$ normal
to the $j$-th $(n-1)$-dimensional face of $P$ (also called a facet),
that are primitive
(i.e., their components have no common factor),
and $P$ may be written as
$$
P=\{y\in\RR^n\,:\, l_j(y):=\langle y,v_j\rangle-\lambda_j\ge0,\quad j=1,\ldots,d\},
$$
with $\lambda_j=\langle p,v_j\rangle\in\ZZ$ with $p$ any vertex on the $j$-th facet,
and $y$ the coordinate on $\RR^n$.
Note that the main results in this article extend to orbifold
toric varieties, since we only make essential use of (i).

The \kahler form $\omega$ is the curvature $(1,1)$ form of a
line bundle $L \to M$. A basis for the  space $H^0(M, L)$ of
holomorphic sections is given by the
monomials
$\chi_{\alpha}(z)=z^\alpha$
with $\alpha \in P$.
More generally, $H^0(M,L)$ generates the coordinate ring
$\oplus_{N=1}^\infty H^0(M,L^N)$, and
each lattice point $\gamma$ in $NP$ corresponds to
a section $\chi_{\gamma}$ of $L^{N} \to M$ defined by
\begin{equation}
\chi_{\gamma}=\chi_{\beta_{1}} \otimes \cdots \otimes
\chi_{\beta_{N}},
\end{equation}
where $\beta_{1},\ldots ,\beta_{N} \in P$ such that $\gamma
=\beta_{1}+ \cdots +\beta_{N}$ (see \cite{STZ}).

We now consider the homogenization (lift to $X$)  of toric \kahler
manifolds.  The lattice points
in $N P$ for each $N \in \mathbb{N}$ correspond in $X$  to the
`homogenized' lattice points $\wh{NP} \subset \mathbb{Z}^{n+1}$ of
the form
\[
\wh{\alpha}^{N}=\wh{\alpha}:=(\alpha_{1},\ldots,\alpha_{n},Np
-|\alpha|),\quad \alpha=(\alpha_{1},\ldots,\alpha_{n}) \in NP \cap
\mathbb{Z}^{n},
\]
where  $p=\max_{\beta \in P \cap \mathbb{Z}^{n}}|\beta|$. 
For simplicity, we generally assume henceforth that $p=1$. 
We also define the cone
$$
\Lambda_{P}:=\bigcup_{N=1}^{\infty}\wh{NP}.
$$ 
Rays $\N\wh{\alpha}$ in this cone define the semiclassical limit.

The monomials $\chi_{\alpha}$ lift to  the CR monomials
$\wh{\chi}_{\wh{\alpha}} (w)\equiv\wh{\chi}_{\alpha} (w), \, w\in X$
(see (\ref{LiftRuleEq})), for
$\wh{\alpha} \in \Lambda_{P}$.  They are joint eigenfunctions of a
quantized torus action on $X$. Let
$$
\xi_j
:=
\frac{\partial}{\partial\theta_j},\quad 1\le j\le n,
$$ 
denote the
Hamiltonian vector fields generating the $\bfT$ action on $M$.  We
use the connection form $\alpha$ to define the horizontal lifts
$\xi_j^h$ of
the Hamiltonian vector fields $\xi_j$:
\begin{equation}
\label{XihjEq}
\pi_* \xi^h_{j} = \xi_j,\;\;\; \alpha(\xi^h_j) = 0,\quad 1\le j\le n.
\end{equation}
Let $\xi_j^* \in \R^n$ denote the element of the Lie algebra of
$\bfT$ which acts as  $\xi_j$ on $M$.
We then define the vector fields $\Xi_j$  by:
\begin{equation}
\label{XiEq}
\Xi_j
:=
\xi^h_j
+
2 \pi \i \langle \nabla\psi \circ\pi, \xi_j^* \rangle
\frac{\partial}{\partial \theta}
=
\xi^h_j + 2\pi \i (\nabla\psi \circ\pi)_j\, \frac{\partial}{\partial
\theta},\quad 1\le j\le n.
\end{equation}

Finally, we define the differential operators (lifted action
operators),

\begin{equation}
\hat{I}_{j}
:=
\Xi_{j},\;\;j=1,\ldots,n,\quad
\hat{I}_{n+1}:=\frac{1}{\i}\frac{\partial}{\partial \theta}
-\sum_{j=1}^{n}\Xi_{j}. \label{Qtorus}
\end{equation}
We recall that $\frac{1}{\i} \frac{\partial}{\partial \theta}$ is
abbreviated by $\D$ and note that $\hat{I}_{n +1}$ is not the same
as $\D$.  Then the monomials $\wh{\chi}_{\wh{\alpha}}$ are the
joint CR eigenfunctions of $(\hat{I}_{1},\ldots,\hat{I}_{n+1})$
for the joint eigenvalues $\wh{\alpha} \in \Lambda_{P}$, i.e.,

\begin{equation} 
\label{jointeigen}
\hat{I}_j \wh{\chi}_{\wh{\alpha}} = \hat{\alpha}_j
\wh{\chi}_{\wh{\alpha}},\;\; \wh{\alpha} \in \Lambda_P ,\;\;\;
\bar{\partial}_b \wh{\chi}_{\wh{\alpha}} = 0, \;\; j=1,\ldots,n+1.
\end{equation}

For simplicity of notation, we denote by $D_{\hat{I}}$ the vector
of first-order operators
\begin{equation} 
\label{DI}
D_{\hat{I}}
:=
\frac1{2\pi\i}\Big(\hat{I}_1,\ldots,\hat{I}_n\Big),
\end{equation}
and use the same notation for the quantized torus action on
$H^0(M, L^N)$ and on $X$.

Although we are primarily concerned with holomorphic sections over
$M$ and their lifts as CR holomorphic functions on $X$, we need to
consider non-CR holomorphic eigenfunctions of the action operators
as well.
 We thus need to consider  the anti-Hardy space
$\overline{\hcal}^2(X)$ of anti-CR functions, i.e. solutions
of $\partial_b f = 0$. A Hilbert basis is given by the
complex-conjugate monomials $\overline{\hat{\chi}}_{\hat{\alpha}}.$

Products of eigenfunctions are also
eigenfunctions. Hence, the orthonormal mixed monomials
$$\hat{\chi}_{\hat{\alpha}, \hat{\beta}}(x)
=\hat{\chi}_{\hat{\alpha}}\overline{\hat{\chi}}_{\hat{\beta}}$$ are
eigenfunctions of eigenvalue $\hat{\alpha} - \hat{\beta}$ for
$\{\hat{I}_{1},\ldots,\hat{I}_{n+1}\}$.
It can be shown \cite{STZ} that
\begin{equation}
\label
{LtwoEigenFnsEq}
L^2(X)
=
\bigoplus_{\hat{\alpha}, \hat{\beta} \in
\Lambda_{P}} \C \hat{\chi}_{\hat{\alpha}, \hat{\beta}}.
\end{equation}
It follows that the joint spectrum of
$(\hat{I}_{1},\ldots,\hat{I}_{n+1})$
on $L^2(X)$ is given by
\begin{equation}
\label{SpectrumActionOpEq}
\hbox{\rm Spec}\, (\hat{I}_{1},\ldots,\hat I_{n+1}) 
= 
\Lambda_{P} - \Lambda_{P}
= 
\ZZ^{n+1}.
\end{equation}

\subsection{Convex analysis}
Here we define some basic notation related to convex functions.
For general background on Legendre duality and convexity we
refer the reader to \cite{Ro}.

A vector
$v\in(\RR^n)^\star$ is said to be a subgradient of a function $f$
at a point $x$ if $f(z)\ge f(x)+\langle v,z-x\rangle$ for all $z$.
The set of all subgradients of $f$ at $x$ is called the
subdifferential of $f$ at $x$, denoted $\partial f(x)$.

The Legendre-Fenchel conjugate of a continuous function $f=f(x)$ on $\RR^n$ is defined by
$$
f^\star(y):=\sup_{x\in\RR^n}\big(\langle x,y \rangle - f(x)\big).
$$
For simplicity,  we will refer to $f^\star$ sometimes as the Legendre dual, or just dual,
of $f$.
An open-orbit \K potential $\psi\in\calH(\bfT)$ is a smooth strictly convex function on $\RR^n$ in logarithmic
coordinates. Therefore its gradient $\nabla\psi$ is one-to-one onto $P=\Im\nabla\psi$ and
one has the following explicit expression for its Legendre dual (\cite{Ro}, or \cite{R}, p. 84--87),
\begin{equation}
\label{LegendreDualityEq}
u(y)=\psi^\star(y)=\langle y,(\nabla\psi)^{-1}(y)\rangle-\psi\circ(\nabla\psi)^{-1}(y),
\end{equation}
which is a smooth strictly convex function on $P$, satisfying
\begin{equation}
\label{LegendreDualityGradientEq}
\nabla u(y)=(\nabla\psi)^{-1}(y).
\end{equation}
Following Guillemin \cite{G}, the function $u$ is called the
symplectic potential of $\i\ddbar\psi$.
The space of all symplectic potentials is denoted by
$\calL\calH(\bfT)$.
Put
\begin{equation}
\label{GuilleminFormulaEq}
u_G:=\sum_{k=1}^d l_k\log l_k.
\end{equation}
A result of Guillemin \cite{G} states that for any
symplectic potential $u$ the difference $u-u_G$ is a
smooth function on $P$ (that is, up to the boundary).
In other words, (\ref{LHTDefEq}) may be rewritten as
\begin{equation}
\label{LHTDefSecondEq}
\calL\calH(\bfT)
=\{u\in C^\infty(P\sm\del P)\,:\, u=u_G+F,\quad\h{\rm with } F\in C^\infty(P)\}.
\end{equation}

\bigskip
\section{Quantizing the Hamiltonian flow of $\dot{\vp_0}$}
\label{QuantizingSection}
\bigskip

In this section $(M,\o)$ is an arbitrary projective \K manifold.
The first step in defining the analytic continuation of $\exp t
X_{\dot\vp_0}$ is to quantize this Hamiltonian flow. We use the
method of Toeplitz quantization \cite{BG,Z2} (see also \cite{Z1}, \S5, 
for some exposition). We may state the
result either in terms of one homogeneous Fourier integral
operator on $L^2(X)$ or as a semi-classical Fourier integral
operator on each of the spaces $L^2_N(X)$ in the decomposition
(\ref{LtwoXBlockDecompositionEq}).

It should be noted that the quantization we use is not unique,
i.e., there exists more than one unitary group of Toeplitz Fourier
integral operators with underlying canonical flow equal to the
Hamiltonian flow of $\dot\vp_0$. Indeed, for any unitary
pseudo-differential 
operator $V =e^{\i  A}$ obtained by
exponentiating a self-adjoint pseudo-differential operator $A$ of
degree zero, and any quantization $U(t)$ of $\exp t
X_{\dot\vp_0}$, the operator $V^* U(t) V$ is another quantization
with the same principal symbol. This lack of uniqueness will be
seen below in the fact that we have more than one version of the
quantization. They are closely related and differ by lower order
terms.

To quantize the classical Hamiltonian, we first quantize the
Hamiltonian  as the
zeroth order Toeplitz operator $\Pi \dot\vp_0 \Pi$ on $H^2(X)$
where $\dot{\vp}_0$ denotes the multiplication operator by
$\dot\vp_0$. It is a bounded Hermitian Toeplitz operator.

\begin{definition}
\label{UtDef} Define the one-parameter subgroup $U(t)$ of unitary
operators on $L^2(X)$ by (cf. (\ref{DDEF}))
\begin{equation}
U(t) = \Pi  e^{\i t \Pi  \D  \dot\vp_0 \Pi} \Pi.
\end{equation}
Its Fourier components are given by
$$
U_N(t) = \Pi_N e^{\i t N\Pi_N \dot \vp_0 \Pi_N} \Pi_N.
$$
\end{definition}

We note that $U(t)$ is not quite the same as $ \Pi e^{\i t  \D
\dot\vp_0 } \Pi, $ which is manifestly the composition of complex
Fourier integral operators. However,  $\Pi_N \dot \vp_0 \Pi_N$ is
the quantization of $\dot\vp_0$. We compose $e^{\i t N \Pi_N \dot
\vp_0 \Pi_N}$ with  $\Pi_N$ to make the operator preserve
$H^0(M,L^N)$. Note that
$
U(t) = \Pi  e^{\i t \Pi  \D  \dot\vp_0 \Pi} = e^{\i t \Pi  \D
\dot\vp_0 \Pi} \Pi.
$

We now  verify that $U(t)$  is a complex Fourier integral  operator
with underlying canonical relation equal to graph of the Hamiltonian flow
at time $t$ of $r \dot{\vp_0}$ on $(\Sigma,\omega_\Sigma)$,
where $r$ and $(\Sigma,\omega_\Sigma)$ are defined in (\ref{SympCone})--(\ref{SympConeTwo}).
This is the content of saying that $U_N(t)$ is a
quantization of the Hamiltonian flow of $\dot{\vp_0}$ on
$(M,\omega_{\vp_0})$.

\begin{prop}
\label{FIOQuantizingHamilton} $U(t)$ is a group of complex Toeplitz Fourier
integral operators on $L^2(X)$ whose underlying canonical relation
is the graph of the time $t$ Hamiltonian flow of $r\dot\vp_0$ on the
symplectic cone $(\Sigma,\omega_\Sigma)$.
\end{prop}

\begin{proof}

We first observe that $U(t)$  is characterized as  the unique
solution of the  ordinary differential equation
$$\frac{d}{dt} U(t) = \left(\i \Pi \D \dot\vp_0 \Pi\right) U(t),
\;\;\; U(0) = \Pi. $$

We use the following result of Boutet de Monvel-Guillemin, whose proof we
sketch later.

\begin{lem}
\label{BGLemma} {\rm (see \cite{BG}, Proposition 2.13)} Let $T$ be a
Toeplitz operator on $\Sigma$ of order $p$. Then there exists a
pseudo-differential operator $Q$ of order $p$ on $X$ such that
$[Q, \Pi] = 0$ and $T = \Pi Q \Pi. $
\end{lem}

We apply Lemma \ref{BGLemma} to  $T = \Pi \dot{\vp_0} \Pi$. Thus,
there exists a zeroth order pseudo-differential operator $Q$ on $X$
with $\sigma_Q |_{\Sigma} = \dot{\vp_0} |_{\Sigma} $ (see  \cite{BG}, 
Theorem 2.9 and Proposition 2.13 for background). Note that here
we identify $\dot\vp_0$ with its lift to $\Sigma\subset T^\star X$.

Since  $\Pi  e^{\i t \Pi  \D  Q \Pi}\Pi$ and $\Pi e^{\i t \D Q }
\Pi$ satisfy the same differential equation
$$
\frac{d}{ dt} W(t)
=
\i \Pi \D Q \Pi W(t)
$$
and have the same initial condition, we have
\begin{equation}
U(t) = \Pi  e^{\i t \Pi  \D  Q \Pi} \Pi = \Pi e^{\i t \D Q } \Pi.
\end{equation}
Here,  we use that $\Pi^2=\Pi$ hence $\Pi Q=\Pi Q \Pi$ and that
$\Pi$ and $\D$ commute.

Now $e^{\i t \D Q}$ is the exponential of a real principal type
pseudo-differential operator of order one on $L^2(X)$ and hence is
a unitary group of Fourier integral operators on $L^2(X)$
quantizing the Hamiltonian flow of $\sigma_{\D Q}$ on $T^\star X$.
Since $\Pi$ is a complex Fourier integral operator whose real
canonical relation is the diagonal in $\Sigma \times \Sigma$
\cite{BSj}, $U(t)$ is also a complex Fourier integral operator.
To complete the proof of the Proposition, it suffices to prove
that the canonical relation of $U(t)$ is the graph of the time $t$
Hamiltonian flow of $r \dot{\varphi_0}$ on $(\Sigma,\omega_\Sigma)$.

Let $\Psi_t$ denote the time $t$ Hamiltonian flow
of $\sigma_{\D}\sigma_Q$ on $(T^\star X,\o_{T^\star X})$.
By the composition theorem for complex Fourier integral operators \cite{MS},
the operator  $\Pi e^{\i t \D Q } \Pi$ is a complex Fourier
integral operator whose canonical relation is the
set-theoretic composition
\begin{equation}
\begin{array}{lll}
\label{CompositionCanonicalRelationTwoEq} \{(v,v)\,:\,v\in\Sigma\}
\circ \{(p,\Psi_t(p)\,:\,p\in T^* X \} \circ
\{(q,q)\,:\,q\in\Sigma \} \cr\cr \qquad\qquad\quad=
\{(m,\Psi_t(m))\,:\,m\in\Sigma \}\cap \Sigma\times\Sigma.
\end{array}
\end{equation}
Here we make use of the fact that the symbol of $\Pi$ is nowhere
vanishing on $\Sigma$  and that of $e^{\sqrt{-1} t \D Q}$ is
nowhere vanishing on the graph of $\Psi_t$.
It only remains to equate (\ref{CompositionCanonicalRelationTwoEq})
with the graph of the time $t$ Hamiltonian flow of
$r\dot\varphi_0$ on $(\Sigma,\omega_\Sigma)$.

Since  $[\Pi,Q]=0$, we have
$$
\Pi e^{\i t\D Q}= \Pi e^{\i t\D Q}\Pi.
$$
This implies that the canonical relations  of both sides in this
equation must be equal.
The canonical relation of the left hand side equals
\begin{equation}
\label{CompositionCanonicalRelationOneEq}
\{(v,v)\,:\,v\in\Sigma\}\circ \{(p,\Psi_t(p)\,:\,p\in T^* X \} =
\{(q,\Psi_t(q)\,:\,q\in\Sigma \}.
\end{equation}
Equating this to (\ref{CompositionCanonicalRelationTwoEq})
it follows that $\Psi_t$ preserves $\Sigma$.
Hence, the Hamiltonian
vector field $X^{T^*X}_{\sigma_{\D}\sigma_Q }$  of $\sigma_{\D} \sigma_Q$
with respect to $\omega_{T^*X}$ is tangent to the symplectic
sub-cone $\Sigma.$

We note that  the symbol of $\D$ is the
Clairaut integral $\sigma_{\D}(x,\xi)=\langle
\xi,\frac{\partial}{\partial \theta}\rangle$.
Since $\alpha\big(\frac{\del}{\del\th}\big)=1$ (see, e.g., \cite{R}, p. 69),
it follows from (\ref{SympCone}) that
$\sigma_{\D}|_{\Sigma} = r.$
Recall also that $\sigma_Q |_{\Sigma} = \dot{\vp_0}|_{\Sigma}$.
Thus, to complete the proof it remains to show that the  restriction of
$\Psi_t$  to $\Sigma$ is the Hamiltonian flow of
\begin{equation}
\label{SigmaDSigmaEq}
\sigma_{\D}\sigma_Q |_{\Sigma} = r \dot{\varphi}_0
\end{equation}
on
$(\Sigma,\o_\Sigma)$.
Let $X^\Sigma_{r\dot{\varphi}_0}$ be the Hamiltonian vector field of
$\sigma_{\D}\sigma_Q |_{\Sigma}$ with respect to $\omega_{\Sigma}$. At a point
of $\Sigma$, we have
$$
\omega_{T^*X}(X^{T^*X}_{\sigma_{\D}\sigma_Q },\, \cdot\,) = d \sigma_{\D}\sigma_Q, \;\;\;
\omega_{\Sigma}(X^\Sigma_{r \dot{\varphi}_0},\, \cdot\,) = d (r \dot{\varphi}_0).
$$
Evaluating these 1-forms on all tangent vectors $Y\in T\Sigma$,
and using (\ref{SympConeTwo}), (\ref{SigmaDSigmaEq}), and that
$X_{\sigma_{\D}\sigma_Q }$  is tangent to $\Sigma$,
we conclude that $X^\Sigma_{r \dot{\varphi}_0} = X^{T^*X}_{\sigma_{\D}\sigma_Q }$.
This completes the proof of Proposition \ref{FIOQuantizingHamilton}.
\end{proof}

\smallskip
\begin{remark}
{\rm
 The fact that the
Hamiltonian vector field $X_{\sigma_{\D} \sigma_Q}$ of
$\sigma_{\D} \sigma_Q$
preserves $\Sigma$ is equivalent to the
fact that the Hamilton vector field  $X_{ \sigma_Q}$ preserves
$\Sigma$.
Indeed,
$$
X_{\sigma_{\D} \sigma_Q} = \sigma_{\D} X_{\sigma_Q} + \sigma_Q
X_{\sigma_{\D}}.$$
Since
$[\D,Q]=0$ and hence $\{\sigma_{\D}, \sigma_Q\}_{\o_{T^\star X}} = 0$
the flows of $X_{\sigma_Q}$ and of $X_{\sigma_{\D}}$ commute.
Hence the Hamiltonian flow of
$X_{\sigma_{\D} \sigma_Q}$ equals the composition
$$
\exp t\sigma_{\D} X_{\sigma_Q}\circ\exp t\sigma_Q
X_{\sigma_{\D}}.
$$
The restricted vector field
$X_{\sigma_{\D}}|_\Sigma=\frac{\partial}{\partial \theta}|_\Sigma$
is equal to $X_{\sigma_{\D}|_\Sigma}$ since the principal
$S^1$-action preserves $\Sigma$ (by (\ref{SympCone}), as it
preserves $\alpha$). Hence its flow always
preserves $\Sigma$. The fact that the flow of $X_{\sigma_Q}$ preserves
$\Sigma$  is proved in \cite{BG},
Proposition 11.4 and the Remark following it. The proof uses the
construction of $Q$ and Toeplitz symbol
calculus, and is therefore similar to the one given above.
}
\end{remark}
\medskip

For the sake of completeness, we briefly sketch a proof of Lemma
\ref{BGLemma}, following the  proof of a similar assertion in 
\cite{GS1}, Theorem 5.8 and Lemma 5.9.

\begin{lem}\label{Q}
Given a smooth real-valued function $q$ on $M$, homogeneous of
degree zero, there exists a self-adjoint pseudo-differential
operator $Q$ such that $[Q, \Pi] = 0$ and such that $\sigma_Q
|_{\Sigma} = q$. \end{lem}

The proof   uses symbol calculus and  spectral theory, all of
which are available in the Toeplitz setting.  The first
observation is that the principal symbol of  $[\Pi, \dot{\vp_0} ]$
vanishes, hence it  is complex Fourier integral operator (or more
specifically Toeplitz operator) of order $-1$. By adding an
operator $Q_{-1}$ to $\dot{\vp_0}$ and using transport equations
for the symbol, one can arrange that the symbols of order $-1$ and
order $ -2$
 of $[\Pi, \dot{\vp_0} + Q_{-1} ]$ equal zero. By repeating
 infinitely often and asymptotically summing the operators, one
 can find $\tilde{Q}$ such that $[\Pi,  \tilde{Q}]=0$
 and $\Pi \dot{\vp_0} \Pi - \Pi \tilde{Q}$
 are smoothing. One then puts $Q = \tilde{Q} + \Pi
 \dot{\vp_0} \Pi - \Pi \tilde{Q} \Pi.$

\bigskip
\section{Two quantiziations of the Hamiltonian flow on a toric manifold}
\label{TwoQuantizationsSection}
\bigskip

In this section we specialize the construction of Section \ref{QuantizingSection} from a general
projective manifold to a toric manifold and study its asymptotic spectrum. We
then give an alternative quantization
of the Hamiltonian flow of $\dot{\vp_0}$ in the special case of
a toric manifold and compare the two quantizations. These results will then
be applied in Section \ref{ConvergenceSection} to complete the proof of Theorem \ref{FirstMainThm}.

Recall from \S\S\ref{ToricBackgroundSubsection} that
the toric monomials $\{\chi_\al(z):=z^\alpha\}_{\a\in NP\cap\Z^{{}^n}}$
are an orthogonal basis of $H^0(M,L^N)$ with respect to any toric-induced
Hilbert space structure on this vector space. Hence any such toric inner product
is completely determined by the $L^2$ norms (up to $N^n/V$), or
``norming constants,"  of the toric monomials---

\begin{equation}
\label{QalphaEq}
\qcal_{h^N}(\alpha):=||\chi_\alpha||^2_{h^N} = \int_{(\C^*)^n}
|z^{\alpha}|^2 e^{-N\psi} \o_h^n.
\end{equation}

\equationspace Here we let $h=e^{-\psi}$ with $\psi\in\calH(\bfT)$.
As in \cite{SoZ1}, we put

$$
\pcal_{h^N}(\alpha,z):=\frac{|\chi_\alpha(z)|^2_{h^N}}{||\chi_\alpha||^2_{h^N}}.
$$

\subsection{The quantization of the Hamiltonian flow using the K\"ahler velocity}
In this subsection we study the one-parameter subgroup $U(t)$ given by
Definition \ref{UtDef} on a toric manifold.

The first observation is that since $\dot\vp_0$ is torus-invariant the
multiplication operator $\dot\vp_0$ preserves the
block decomposition (\ref{LtwoXBlockDecompositionEq}). Therefore the
toric monomials diagonalize the Toeplitz operators
$\Pi_N\dot\vp_0\Pi_N$, that is,
\begin{equation}
\Pi_{N}\dot{\vp}_0 \Pi_{N} \chi_{\alpha}
=
\mu_{N,\a} \chi_{\alpha},
\end{equation}
for some real numbers $\{\mu_{N,\a}\}_{\a\in NP\cap\ZZ^n}$.
Since $\{\chi_{\alpha}\}_{\a\in NP\cap\ZZ^n}$ are orthogonal with
respect to a toric inner product we have
\begin{equation}
\label{EIG}
\mu_{N,\a}
=
\frac{1}{\calQ_{h_0^N}(\alpha)}
\int_M \dot\vp_0 |\chi_{\alpha}|^2_{h_0^N} \o^n_{\vp_0}.
\end{equation}

Hence we have the the following expression for the level $N$ quantum
analytic continuation potential induced by $U(\i s)$:
\begin{equation}
\vp_N(s,z)
=
\frac{1}{N} \log U_N(-\i s,z,z)
=
\frac1N\log\sum_{\a\in NP\cap\ZZ^n}
e^{sN \mu_{N,\a}}
\frac{|\chi_{\alpha}(z)|^2_{h_0^N}}{\calQ_{h_0^N}(\alpha)}.
\end{equation}

\subsection{An alternative quantization using the symplectic potential}

We now introduce a second quantization in the special case of a
toric manifold for which the eigenvalues are special values of the
velocity of the symplectic potential.  In effect, it is an
explicit construction of the operator $Q$ in Lemma \ref{BGLemma}, at
least to leading order (which is sufficient for our purposes).

\begin{definition}
Define the one-parameter subgroup $V(t)$ of unitary operators on $L^2(X)$ by
$$
V(t) = \Pi e^{-\i t\D \dot u_0 (D_{\hat I}\D^{-1})}\Pi .
$$
Its Fourier components are given by
$$
V_N(t) = \Pi_N e^{-\i tN  \dot u_0 (N^{-1} D_{\hat I})} \Pi_N.
$$
\end{definition}

It follows that the level $N$ quantum analytic continuation potential 
induced by $V(\i s)$ is given by
\begin{equation}
\label{tildevpEq} \begin{array}{lll}
\tilde \vp_N(s,z)
& := &
\displaystyle
\frac{1}{N} \log V_N(-\i s,z,z)\\ && \\ & =&
\displaystyle
\frac1N\log\sum_{\a\in NP\cap\ZZ^n} e^{-sN \dot u_0(\a/N)}
\frac{|\chi_{\alpha}(z)|^2_{h_0^N}}{\calQ_{h_0^N}(\alpha)}.
\end{array} \end{equation}

In order to relate $\tilde \vp_N$ to the actual quantum analytic
continuation potentials $\vp_N$ the following fact is crucial.

\begin{prop}
\label{SecondQuantizationThm}
The sequence of unitary operators $\{V_N(t)\}_{N\ge1}$
is a semi-classical complex Toeplitz Fourier integral operator quantizing the
time $t$ Hamiltonian flow of $\dot{\vp}_0$ on $(M,\o_{\vp_0})$.
\end{prop}

\noindent
We note that, equivalently, Proposition \ref{SecondQuantizationThm} could be
stated in `homogeneous' notation, that is, in an identical manner 
to Proposition \ref{FIOQuantizingHamilton} with $U(t)$ replaced by $V(t)$.

\begin{proof} It is convenient to lift to the circle bundle $X$
and use the full spectral theory of the action operators of
\S\S\ref{ToricBackgroundSubsection}.

We observe that $\Pi \dot{u_0}
(D_{\hat I} \D^{-1}) \Pi $ is defined by the Spectral Theorem to
be  the operator on 
$$
H^2(X)\sm\CC
=
\bigoplus_{N\in\NN}H^2_N(X)
$$ 
whose eigenfunctions are the same as
the joint eigenfunctions of the quantum torus action, i.e., the
lifted monomials 
$$
\{\hat{\chi}_{\hat\a}\,:\, \hat\a\in\Lambda_P\},
$$
and whose corresponding
eigenvalues are 
$$
\big\{\dot u_0(\a/N)\,:\, N\in\NN,\;\a\in NP\cap \ZZ^n
\big\}.
$$
However, in order to apply classical results concerning operators
of the form $e^{\i t P}$ where $P$ is a real first-order pseudo-differential
operator of principal type we need to replace $\dot{u_0}(D_{\hat I}\D^{-1})$
with an operator defined on all of $L^2(X)$. 
Yet, since eventually we pre- and post-compose 
with $\Pi$, we are ultimately only interested in 
the restriction to $H^2(X)\sm\CC$ of the extended 
operator. Hence we would like the  extended operator to 
coincide with $\dot{u_0}(D_{\hat I}\D^{-1})$ on $H^2(X)\sm\CC$.
This is the purpose of the following Lemma.

\begin{lem}
\label{PsiDOOrderzeroLemma}
There exists a pseudo-differential operator $R$ of order zero on $L^2(X)$
such that 
\begin{equation}
\label{RHardySpaceEq}
R|_{H^2(X)\sm\CC}=\dot{u_0}(D_{\hat I}\D^{-1})|_{H^2(X)\sm\CC}.
\end{equation}
\end{lem}

\begin{proof}
There are two obstacles to defining 
$\dot{u_0}(D_{\hat I}\D^{-1})$ on all of $L^2(X)$. 
First, according to (\ref{LtwoEigenFnsEq})--(\ref{SpectrumActionOpEq})
we need to define $\dot u_0$ on $\RR^n$, while originally it is only
defined on $P$.
Second, the operator $\D^{-1}$ is only defined on the orthocomplement of
the invariant functions on $X$ for the $S^1$ action. 
The non-constant CR functions are orthogonal to the invariant functions, so $\Pi
\dot{u_0} (D_{\hat{I}} \D^{-1}) \Pi $ is well-defined on $H^2(X)\sm\CC$. 
But we wish to extend $\dot{u_0} (D_{\hat{I}} \D^{-1}) $ outside the
Hardy space.

To deal with the first point, note that
since $\dot{u_0}$ is smooth up to the boundary of $P$, 
we may assume it is defined in some neighborhood 
of $P$ in $\RR^n$, and then multiply it by a smooth
cutoff function $\eta$ equal to $1$ in a neighborhood of $P$
and with compact support in $\RR^n$.
Then $\eta \dot{u_0}$ is a smooth function of compact support
in $\RR^n$, and $\eta\dot{u_0}(D_{\hat I}\D^{-1})\equiv
(\eta\dot{u_0})(D_{\hat I}\D^{-1})$ is well-defined
on $(\ker\D)^\perp\subset L^2(X)$. As noted above, 
$$
\bigoplus_{N\in\NN}H^2_N(X)
=
H^2(X)\sm\CC\subset (\ker\D)^\perp,
$$
and since 
$$
\h{\rm Spec}\,D_{\hat I}\D^{-1}|_{H^2(X)\sm\CC}\subset P,
$$ 
we have 
\begin{equation}
\label{FirstCutOffEq}
\eta\dot{u_0}(D_{\hat I}\D^{-1})|_{H^2(X)\sm\CC}
=
\dot{u_0}(D_{\hat I}\D^{-1})|_{H^2(X)\sm\CC}.
\end{equation}

We now turn to the second point.
There are several ways of handling it; 
in addition to the construction that follows we mention
two other possibilities in Remark \ref{RemarkThirdOption} below.
For any $\epsilon>0$ let 
$\gamma_\epsilon=\gamma_\epsilon(\sigma_{\hat I},\sigma_\D)\in C^\infty(T^\star X\sm\{0\})$ denote a 
homogeneous 
frequency cut-off, equal to $1$ in an open conic neighborhood  
\begin{equation}
\label{CutOffSigmaDEq}
\Big\{(x,\xi)\in T^\star X\sm\{0\}\,:\,
|\sigma_{\D} | 
< 
\epsilon 
\big(
|\sigma_{D_{\hat I}}|^2+\sigma_{\D}^2
\big)^{1/2}
\Big\}
\end{equation}
of the set $\{\sigma_{\D} = 0\}$,
and vanishing on
$$
\Big\{(x,\xi)\in T^\star X\sm\{0\}\,:\,
|\sigma_{\D} | 
>
2\epsilon 
\big(
|\sigma_{D_{\hat I}}|^2+\sigma_{\D}^2
\big)^{1/2}
\Big\}
$$
(note that $n+1$ of the vertical directions in $T^\star X$ are not involved).
Let $\beta\in \ZZ^{n+1}$, and let $\chi_\beta\in L^2(X)$ be the associated monomial. 
Denote by
$\gamma_\epsilon(D_{\hat I},\D)$ the Fourier multiplier associated to 
$\gamma_\epsilon$, 
namely such that 
$$
\gamma_\epsilon(D_{\hat I},\D){\chi}_{\beta}(w) 
= 
\gamma_\epsilon(\beta){\chi}_{\beta}(w).
$$
This defines $\gamma_\epsilon(D_{\hat I},\D)$ on $L^2(X)$ (see (\ref{LtwoEigenFnsEq})).
Let $I$ denote the identity operator on $L^2(X)$.
Then $I-\gamma_\epsilon(D_{\hat I},\D)$ is a pseudo-differential operator
with 
$$
\ker\D\subset \ker(I-\gamma_\epsilon(D_{\hat I},\D)),
$$
and 
$$
R_\epsilon:=\eta \dot{u_0}(D_{\hat{I}} \D^{-1}) (I -\gamma_\epsilon(D_{\hat I},\D))
$$ 
is a pseudo-differential operator of order zero,
defined on all of $L^2(X)$.

To complete the proof of the Lemma, we will prove
that (\ref{RHardySpaceEq}) holds for $R:=R_\epsilon$,
for any $\epsilon>0$ small enough.

Let $\hat\alpha\in\Lambda_P$ with $\chi_\alpha\in H^0(M,L^N)$,
$\alpha\in NP\cap \ZZ^n$.
We claim that for small enough $\epsilon>0$ in
(\ref{CutOffSigmaDEq}) we have
$$
\gamma_\epsilon(D_{\hat I},\D)\hat{\chi}_{\hat{\alpha}}(w) 
= 
\gamma_\epsilon(\alpha,N)\hat{\chi}_{\hat{\alpha}}(w) 
= 
0.
$$ 
For the second equality, note that for $(w, r \alpha(w)) \in \Sigma$, we have 
$\gamma_\epsilon(w, r \alpha(w)) = 0$ unless 
\begin{equation}
\label{OpenConeRestrictedToSigmaEq}
r 
\le 
2\epsilon r(|\nabla\psi_0 \circ\pi (w)|^2+1)^{1/2},
\end{equation}
where $\pi:X\ra M$ is the bundle projection map.
For $r>0$ equation (\ref{OpenConeRestrictedToSigmaEq})
cannot hold if we take $\epsilon$ such that
\begin{equation}
\label{EpsilonRangeEq}
0<\epsilon 
< 
\frac1{2\sqrt{\sup_{y\in P}|y|^2+1}},
\end{equation}
since $\nabla\psi_0 \circ \pi (w) \in P$ 
(and $P$ is a bounded set in $\RR^n$).  
This proves the claim, for $\epsilon$ satisfying (\ref{EpsilonRangeEq})
(note that in the proof of the last assertion, instead
of working in `homogeneous' notation, 
we could have replaced $r>0$ by $N\in\NN$ and $r\nabla\psi_0\circ\pi(w)$
by $\a\in NP$).
It follows that
$$
I=
I-\gamma_\epsilon(D_{\hat I},\D),
\quad \h{\ on $H^2(X)\sm\CC$}.
$$
Together with (\ref{FirstCutOffEq})
this proves that
$$
\dot{u_0} (D_{\hat{I}} \D^{-1})\Pi  
= 
\eta\dot{u_0}(D_{\hat{I}}\D^{-1})(I-\gamma_\epsilon(D_{\hat I},\D))\Pi,
\quad\h{\ for each $\epsilon$ satisfying (\ref{EpsilonRangeEq})},
$$
as desired.
\end{proof}

The following Lemma is the concrete realization of Lemma \ref{BGLemma}
in the setting of toric \kahler manifolds.

\begin{lem}
\label{SYMBOLLemma} 
Let $\epsilon$ satisfy (\ref{EpsilonRangeEq}) and let
$R:=\eta\dot{u_0} (D_{\hat{I}}\D^{-1})(I-\gamma_\epsilon(D_{\hat I},\D))$.
The operator $\Pi R\Pi$ is a Toeplitz operator of order zero and its symbol
is given by
$$
\sigma_R(w,\xi) = \dot{u}_0\circ\nabla\psi_0 \circ \pi(w) =
-\dot\vp_0\circ\pi(w),\quad (w,\xi)\in \Sigma,
$$
where $\pi:X\ra M$ is the projection onto the base.

\end{lem}

\begin{proof}  
As noted in the proof of Lemma \ref{PsiDOOrderzeroLemma}, 
the symbol of $I-\gamma_\epsilon(D_{\hat I},\D)$ equals one on 
$\Sigma$. In addition, when restricting to $\Sigma$,
the operator $\dot{u_0} (D_{\hat{I}}\D^{-1})$ has
a well-defined symbol, equal to the symbol of
$\eta\dot{u_0} (D_{\hat{I}}\D^{-1})$, restricted to $\Sigma$.

On $\Sigma$, the symbols of the vector fields $\xi_j^h$ 
(see (\ref{XihjEq})) are the Clairaut integrals
$$
\sigma_{\xi_j^h}(w, r \alpha(w)) = \alpha_w(\xi_j^h) = 0.
$$
Hence, on $\Sigma$, the symbol of $\hat{I}_j, 1\le j\le n$,   
is that of the second term in (\ref{XiEq}): 
$2\pi \i r (\nabla\psi_0 \circ\pi)_j$. By the normalization
of (\ref{DI}) therefore
$$
\sigma_{D_{\hat I}}(w,r\a(w))=r\nabla\psi_0\circ\pi(w).
$$
Since $\sigma_{\D^{-1}}(w,r\a(w))=1/r$ (see the proof of
Proposition \ref{FIOQuantizingHamilton}), it follows
that the symbol of $\dot{u_0}(D_{\hat{I}} \D^{-1})$, restricted to $\Sigma$,
is $\dot{u_0}(\pi^* \nabla\psi_0)$ and thus equals the stated
Hamiltonian $\sigma_R$. It is the lift of the Hamiltonian 
$H(z) = \dot{u}_0 \circ \nabla\psi_0(z)$ to the cone 
$\Sigma=\Sigma_{h_0}$. 
\end{proof}

We may now conclude the proof of Proposition \ref{SecondQuantizationThm}. 
Indeed, from Lemma \ref{PsiDOOrderzeroLemma} we have that
$$
V(t)=\Pi e^{-\i t\D R}\Pi.
$$
Since $\D R$ is a real principal type pseudo-differential operator 
of order $1$, it follows that $e^{\i t\D R}$ is a unitary
Fourier integral operator whose canonical relation is given by
$$
C=
\{
((w,\xi), (v,\zeta))\,:\, (w,\xi), (v,\zeta)\in T^\star X\sm\{0\},
(w,\xi)=\exp t X^{T^\star X}_{\sigma_{\D R}}(v,\zeta)
\}
$$
(see, e.g., \cite{DG}, Theorem 1.1, or \cite{H}, Theorem 29.1.1; note that
ellipticity is not essential).
It follows then from Lemma \ref{SYMBOLLemma} that the canonical relation
of $V(t)$ is given by the time $t$ flow-out of $\Sigma$ under the 
flow of the Hamiltonian $-\sigma_{\D R}=r\pi^*\dot\vp_0$ with respect to $(T^\star X,\o_{T^\star x})$. 
As shown in the proof of Proposition \ref{FIOQuantizingHamilton} this coincides 
with the time $t$ flow of $\Sigma$ under the flow of the same Hamiltonian 
with respect to $(\Sigma,\o_{\Sigma})$. 
Finally, the corresponding statement for the operators $V_N(t)$ asserted in
the Proposition follows by `de-homogenization', since when restricting
to $H^0(M,L^N), N\in\NN,$ the operator $\D$ simply acts by multiplication by $N$,
and so we may replace $r$ by the constant $N$,
 concluding the proof.
\end{proof}

\smallskip
\begin{remark}
\label{RemarkThirdOption}
{\rm
In place of $\gamma_\epsilon(D_{\hat I},\D)$ we could use at least two
other constructions.

First, we could replace
$\dot u_0(D_{\hat{I}} \D^{-1})$ with the globally  well-defined operator 
$$
\dot u_0(D_{\hat{I}}(I+|\D|^2)^{-1/2}).
$$ 
We have that
$s_{D_{\hat I}} (1 + |s_{\D}|^2)^{-1/2}
\approx 
s_{D_{\hat I}}s_{\D}^{-1}$
asymptotically as $r \to \infty$ in $\Sigma$.  
Therefore, the principal symbols (that are homogeneous of degree $0$) 
of the associated Toeplitz operators are equal,
$
\sigma_{\dot u_0(D_{\hat{I}}(I+|\D|^2)^{-1/2})}|_\Sigma
=
\sigma_{\dot u_0(D_{\hat{I}} \D^{-1})}|_\Sigma.
$
However, this new operator has different 
eigenvalues (although this would not matter later in
proving Theorem \ref{FirstMainThm}, since 
$\frac\a N-\frac\a{\sqrt{1+N^2}}=O(1/N^3)$).

Alternatively, we could use the
orthogonal projection $\Pi_0^{\perp}$
onto the orthogonal complement of the invariant functions. 
This is finer than $1-\chi(\D)$ since its symbol vanishes
outside $\{\sigma_\D=0\}$ and not just outside an open
cone containing it.
Note that $\{\sigma_\D=0\}\subset T^\star X$ is the `dual' 
of the horizontal bundle over $X$ in $TX$ (with respect
to the connection $\alpha$), and does not intersect $\Sigma$
which is itself dual to the vertical bundle.
Since
the spectrum of $\D $ lies in $\Z$, 
the operator
$
\eta\dot{u_0} (D_{\hat{I}}
\D^{-1}) \Pi_0^{\perp} 
$ 
is well-defined on all of $L^2(X)$. From
the formula 
$
\Pi_0 
= 
\frac{1}{2 \pi} \int_0^{2 \pi} e^{\i \theta\D}d \theta,
$ 
we see that $\Pi_0$ is a zeroth order Fourier
integral operator whose canonical relation is 
$$
\mskip-200mu C=\{((w,\xi),(w',\xi')) \in T^\star X\sm\{0\} \times T^\star X\sm\{0\}\,:\,
$$
$$
\mskip100mu\;\sigma_\D(w,\xi)=0, \; w' = e^{\i \theta} w,\; \xi'=e^{\i\theta}\xi,\, \h{\rm\ for 
some $\theta\in[0,2\pi)$}\,\}.
$$ 
Since $\Pi_0^{\perp} = I - \Pi_0$ is
also a Fourier integral operator,  $ (\eta \dot{u_0}) (D_{\theta}
\D^{-1}) \Pi_0^{\perp}$ is a well-defined Fourier integral
operator and 
$
\Pi \eta \dot{u_0}(D_{\theta} \D^{-1})
\Pi_0^{\perp}\Pi = \Pi\eta \dot{u_0}(D_{\theta} \D^{-1}) \Pi,
$
as Toeplitz Fourier integral operators.
We can then compute the symbol of the Fourier integral operator
$\eta \dot{u_0}(D_{\theta} \D^{-1})\Pi_0^{\perp}$
using the composition theorem \cite{H} and obtain the same answer as in the proof
of Lemma \ref{SYMBOLLemma}, 
since $C$ has an empty composition with the canonical relation of $\Pi$, 
$\Delta_{\Sigma\times\Sigma}\subset\Sigma\times\Sigma$,
as $\sigma_\D|_\Sigma\ne0$.
}
\end{remark}
\medskip\smallskip

\subsection{Comparison of the quantizations}

The reason we introduced the second quantization is that its eigenvalues are explicitly given
in terms of the symplectic potential. Since both of our unitary one-parameter subgroups quantize
the same Hamiltonian flow, we obtain the following relation between their spectra.

\begin{lem}
\label{muNaLemma}
We have 
$$
\mu_{N,\a}= -\dot u_0(\frac{\alpha}{N})+O(1/N).
$$ 
More precisely,
there exists $C>0$ independent of $N$ or $\alpha\in NP$ such that
$$
|\mu_{N,\a}+\dot u_0(\alpha/N)|\le \frac CN.
$$
\end{lem}

\begin{proof}
By Lemma \ref{SYMBOLLemma},  $\Pi_N \dot{\vp}_0 \Pi_N$ and $- \Pi_N
\dot{u}_0 (D_{\hat I}\D^{-1}) \Pi_N$ are zeroth order Toeplitz
operators with the same principal symbols. Hence they differ by a
Toeplitz operator of order $-1$. 
Let $\chi_\a\in H^0(M,L^N)$.
It follows that
$$
\mu_{N,\a}
=
\frac{\langle \Pi_N\dot{\vp}_0 \Pi_N \chi_{\alpha}, \chi_{\alpha}\rangle}{Q_{h_0^N}(\alpha)}
=
-\frac{\langle \Pi_N \dot{u}_0(D_{\hat{I}} \D^{-1}) \Pi_N \chi_{\alpha}, \chi_{\alpha}
\rangle}{Q_{h_0^N}(\alpha)}
+
O\Big(\frac{1}{N}\Big),
$$
proving the Lemma.
\end{proof}

\smallskip
\begin{remark}
{\rm We briefly relate the Lemma above to some calculations in
\cite{SoZ1,Z4}.
First, $\Pi_N \dot{u_0}(D_{\hat{I}} \D^{-1}) \Pi_N(z,z)$ is precisely the
kind of  Bernstein polynomial discussed in \cite{Z4}. There
it was shown that
$$
\Pi_N \dot{u_0}(D_{\hat{I}} \D^{-1}) \Pi_N(z,z)
=
N^n
\dot{u_0}(\nabla\psi_0(z)) + O(N^{n-1}) = - \dot{\vp}_0(z) N^n +O(N^{n-1}).
$$
In the language of Berezin-Toeplitz operators,
this shows that $\Pi_N \dot{u_0}(D_{\hat{I}} \D^{-1}) \Pi_N$ and $- \Pi
\dot{\vp_0} \Pi_N$ have the same Berezin symbol. There is an
invertible (Berezin) transform from the Toeplitz symbol
(calculated in Lemma \ref{SYMBOLLemma}) and the Berezin symbol, so this
gives another proof of Lemma \ref{muNaLemma} (noting the $N^n$ 
factor in passing from $H^2_N(X)$ to $H^0(M,L^N)$, see
\S\S\ref{KahlerQuantSubsection}).

One could also evaluate the eigenvalues directly by pushing
forward the eigenvalue integral to $P$ via the moment map
$\nabla\psi_0$ and using equations
(\ref{CoordsOpenOrbitEq}),(\ref{LegendreDualityEq}), and the
identity
$$
(\nabla\psi_0)_\star \o_{h_0}^n=dy,
$$
giving
\begin{equation}
\label{EigIntEq} \mu_{N,\a} = \frac{1}{\calQ_{h_0^N}(\alpha)}
\int_P -\dot{u_0}(y) e^{N (u_0(y)+\langle \frac{\alpha}{N}-y,
\nabla u_0(y) \rangle)} dy.
\end{equation}
Integrals similar to this one are calculated asymptotically in
\cite{SoZ1}. For instance, when $d(\frac{\alpha}{N},
\partial P)>\frac{\log N}{N}$ we may apply the steepest
descent  method to (\ref{EigIntEq}). There is a unique critical
point $y=\frac{\alpha}{N}$ and
\begin{equation}
\label{eigasymp}
\mu_{N,\a} =
-\frac{1}{Q_{h^N}(\a)}\dot{u_0}(\a/N)Q_{h^N}(\alpha) +
O(\frac{1}{N}).
\end{equation}
The evaluation in the boundary zone is more complicated and can be done by Taylor
expansions centered at the boundary.}

\end{remark}
\medskip

\bigskip

As a corollary of Lemma \ref{muNaLemma} we have a corresponding
result on the level of potentials. Let $h_s = e^{-\vp_s} h_0$.

\begin{cor}
\label{PotentialComparisonLemma}
There exists a constant $C>0$ independent of $N$ or $z$ such that
$$
|
\tilde\vp_N(s,z)
-
\vp_N(s,z)
|
\le
\frac{Cs\log N}N.
$$
\end{cor}

\begin{proof}
By Lemma \ref{muNaLemma}, we have for some uniformly bounded function $R(N,\alpha)$ that
\begin{equation}
\label{varphiNtzFirstEq}
\begin{array}{lll}
\vp_N(s,z)
& =
\dis
\frac{1}{N}
\log
\sum_{\a} e^{sN \mu_{N,\a}}
\frac{ |\chi_\a|^2_{h_0^N} } { \calQ_{h_0^N}(\a) }
\\ \\
& =
\dis
\frac{1}{N}
\log
\sum_{\a} e^{-sN \dot u_0(\a/N)+sR(N,\a)}
\frac{ |\chi_\a|^2_{h_0^N} } { \calQ_{h_0^N}(\a) }.
\end{array}
\end{equation}
The result now follows by comparing this with the
expression (\ref{tildevpEq}) for $\tilde\vp_N(s,z)$.
\end{proof}

Equation (\ref{varphiNtzFirstEq}) leads to a heuristic
proof of Theorem  \ref{FirstMainThm}: According to  \cite{SoZ1}
(Propositions 3.1 and 6.1),
\begin{equation}
\label{QNAsympEq}
\calQ_{h_0^N}(\a)=F(\a,N)e^{Nu_0(\a/N)}/N^{C(\a,n)},
\end{equation}
where $C(\a,n)$ and $F(\a,N)$ are some uniformly bounded functions.
Substituting this into (\ref{varphiNtzFirstEq}) we obtain
\begin{equation}
\label{varphiNtzSecondEq}
\begin{array}{lll}
\vp_N(s,z)
& =
\dis
\frac{1}{N}
\log
\sum_{\a}
e^{N(\langle x,\a/N\rangle-\psi_0(x)-u_s(\a/N))+sR(N,\alpha)}+O(\log N/N).
\end{array}
\end{equation}
Intuitively, the  leading order logarithmic asymptotics are given by the value of the principal part
of the exponent,
$$
\langle x,\a/N\rangle-\psi_0(x)-u_s(\a/N),
$$
at its  maximum (over $\a\in NP\cap\ZZ^n$).
But this value is $u_s^\star(x) - \psi_0(x)$, as stated in Theorem 1.

In the next section we give a rigorous proof.

\bigskip
\section{Convergence of the quantization to the IVP geodesic}
\label{ConvergenceSection}
\bigskip

We now complete the proof of Theorem \ref{FirstMainThm}.
We study the large $N$ limit of
\begin{equation}
\vp_N(s,z)
=
\frac1N\log\sum_{\a\in NP\cap\ZZ^n}
e^{sN \mu_{N,\a}}\calP_{h_0^N}(\a,z),
\end{equation}
First note that the $C^2(M\times[0,T^\cvx_\span))$ convergence is a direct corollary of the work
of Song-Zelditch
\cite{SoZ1}. Indeed for all $T<T^\cvx_\span$ the geodesic is smooth and so we may consider it as a
smooth
endpoint geodesic connecting $\vp_0$ to $\vp_T\in\H(\bfT)$.
It thus remains to prove
that $\vp_N(s, z)$ converges to $\vp_s(z)$ in
$C^0(M\times[0,T])$ for all $T>0$. The argument
here is somewhat different than the corresponding
$C^0$ convergence results in \cite{RZ1,SoZ1,SoZ2}
due to the fact that our limit is less regular,
namely only Lipschitz. Due to this reduced regularity
we may not apply asymptotic expansions for families
of smooth Bergman metrics (for example
the asymptotic expressions for the norming constants
or the peak values of the monomials
derived in \cite{SoZ1}), nor can we use the standard
asymptotic expansion of the Bergman kernel
\cite{Z3}. Finally, we also do not have a genuine moment map.

According to Corollary \ref{PotentialComparisonLemma},  in order
to prove convergence of $\vp_N(s,z)$ to $\vp_s(z)$ it will be
enough to consider the difference

\begin{equation}
\label{ENDefEq}
E_N(s,z)
:=
\tilde\varphi_N(s,z)-\varphi_s(z)
=
\frac{1}{N}
\log \sum_{\alpha \in NP\cap\ZZ^n}
e^{-sN \dot u_0(\a/N) }
\frac{|\chi_\a(z)|^2_{h_s^N}}{\calQ_{h_0^N}(\a)}
\end{equation}
Theorem \ref{FirstMainThm} will then follow from the following
result.

\begin{lem}
\label{CZeroLemma}
For every $T>0$, we have
\begin{equation}
\lim_{N\rightarrow \infty} \sup_{s \in [0,T]}  || E_N(s,z) ||_{C^0(M)}=
0.
\end{equation}
\end{lem}

\proof
Whenever $T< T_{\span}^\cvx$ the result follows directly from (\ref{QNAsympEq})
and the asymptotic expansion of the Bergman kernel: applying
(\ref{QNAsympEq}) to $h_0$, using the explicit formula for $u_s$ and then
applying (\ref{QNAsympEq}) to $h_s$, we obtain
$$
E_N(s,z)
=
\frac{1}{N}
\log \sum_{\alpha \in NP\cap\ZZ^n}
\frac{|\chi_\a(z)|^2_{h_s^N}}{\calQ_{h_s^N}(\a)}+O(\log N/N),
$$
and this is $O(\log N/N)$ by the asymptotic expansion of the Bergman kernel \cite{Z3}.
Here by $O(\log N/N)$ we mean a quantity that is bounded from above
and below by $\pm C\frac{\log N}N$ where $C$ may depend on the Cauchy
data and on $T$.

Assume now that $T\ge T_\span^\cvx$. First, we have
\begin{equation}
\label{SummandFirstEq}
e^{-sN \dot u_0(\a/N) }
\frac{|\chi_\a(z)|^2_{h_s^N}}{\calQ_{h_0^N}(\a)}
=
e^{-sN \dot u_0(\a/N) }
\frac{e^{\langle x,\alpha\rangle-N\psi_s}}
{\calQ_{h_0^N}(\a)}.
\end{equation}
From the definition of the Legendre transform we
obtain that this is bounded from above by
$$
e^{-sN \dot u_0(\a/N)+\langle x,\alpha\rangle+Nu_s(\a/N)-N\langle x,\alpha/N\rangle}
/\calQ_{h_0^N}(\a)
=
e^{N u_0(\a/N)}/\calQ_{h_0^N}(\a).
$$
Applying (\ref{QNAsympEq}) to $h_0$ and using the fact 
that $\dim H^0(M,L^N)$ is polynomial in $N$
we obtain that
$$
E_N(s,z)\le O(\log N/N).
$$

We now turn to proving a lower bound for $E_N(s,z)$ when $T\ge T_\span^\cvx$.
Rewrite (\ref{SummandFirstEq}) as
$$
e^{\langle x,\alpha\rangle-N\psi_s-Nu_s(\a/N)}F(\a,N)N^{-C}.
$$
A lower bound for $E_N(s,z)$ will follow once we find one summand in (\ref{ENDefEq})
that is not decaying to zero too fast.
More precisely, we will seek $\tilde N=\tilde N(s,x)$ and
one $\alpha=\alpha(N,s,x)\in NP\cap\ZZ^n$ for each $N>\tilde N$, for which
$$
e^{\langle x,\alpha\rangle-N\psi_s-Nu_s(\a/N)}
\ge
e^{-CN^{1-\epsilon}},
$$
for some $\epsilon>0$.

Fix $x\in\RR^n$ (recall $|z|^2=e^x$). The \K potential $\psi_s$ is
defined on all of $\RR^n$  and satisfies
\begin{equation}
\label{LegendreTransformInequalityEq}
\psi_s(x)\ge \langle x,y\rangle-u_s(y), \quad \forall y\in P,
\end{equation}
with equality if and only if $y\in\partial \psi_s(x)$ (see
\cite{Ro}). 
Let $y_1\in P$ satisfy equality in (\ref{LegendreTransformInequalityEq}). 
It exists, since the supremum in
$$
\psi_s(x)=\sup_{y\in P}[\langle x,y\rangle -u_s(y)],
$$
in necessarily achieved and finite ($P$ is compact and $u_s$ is bounded);
hence by convexity of $\psi_s$ we have $\partial\psi_s(x)\ne\emptyset$,
and one may choose then $y_1\in\partial\psi_s(x)$.
Then we need to find $\tilde N=\tilde N(s,x)$ and $\alpha=\alpha(N,s,x)$ such that
$$
e^{N(\mskip1mu\langle x,\alpha/N-y_1\rangle+u_s(y_1)-u_s(\a/N))}
\ge
e^{-CN^{1-\epsilon}}, \quad \h{\rm\ for each $N>\tilde N$}.
$$
In fact we will derive such an estimate
where the right hand side is $e^{-C\log N}$.
First, we need the following result concerning $\del\psi_s(x)$.

\begin{claim}
\label{ClaimInteriorPolytope}
Let $x\in\RR^n$ and let $y_1\in\del\psi_s(x)$. Then
 $y_1\in P\sm \del P$.
\end{claim}

\begin{proof}

Note that by duality $x\in\del u_s(y_1)$ (this holds
even though $u_s$ need not be convex, 
see \cite{HL}, Theorem 1.4.1, p. 47), and in particular
$\del u_s(y_1)\not=\emptyset$.
Therefore, 
it suffices to show that 
$\lim_{y\ra\del P}|\nabla u_s(y)|=\infty$,
since that will imply that 
$\del u_s(y)=\emptyset$ whenever $y\in \del P$.

Let $\{w_i\}\subset P\sm\del P$ be a sequence converging to
$y\in \del P$. Assume without loss of generality 
that $l_1,\ldots,l_n$ provide a coordinate chart 
in a neighborhood of $y$ in $P$.
Using Guillemin's formula (\ref{GuilleminFormulaEq}),
in these coordinates the gradient
of $u_s$ takes the form 
$(\log l_1+h_1,\ldots,\log l_n+h_n)$,
where $h_j\in C^\infty(P),\,j=1,\ldots,n$.
It then follows that
$\lim_{y\ra\del P}|\nabla u_s(y)|=\infty$,
as desired.
\end{proof}

The points $\{\a/N\}_{NP\cap\ZZ^n}$ are $C/N$-dense in $P$, where $C>0$ is
some uniform constant.
Hence, for each of the $2^n$ orthants in $\RR^n$ there exists a point
$\a/N$ that is $C/N$-close to $y_1$
and such that the vector $\a/N-y_1$ is contained in that orthant.
Now let $\tilde N$ be chosen large enough so
that $\hbox{dist}(y_1,\partial P)>C/\tilde N$
(possible by Claim \ref{ClaimInteriorPolytope}).
Further, let $\tilde N$ be chosen so large such that 
we may find $\alpha_1=\a_1(\tilde N)$ such that $\a_1/\tilde N\in P\setminus\partial P$ and
\begin{equation}
\label
{AlphaOneDistIneqDefEq}
\hbox{dist}(\a_1/\tilde N,\partial P)>C/\tilde N,
\end{equation}
and also
\begin{equation}
\label{AlphaOneIneqDefEq}
\langle\a_1/\tilde N-y_1,x\rangle
\ge 
0, 
\quad 
\hbox{and }\quad \frac C{2\tilde N} 
\le 
|\a_1/\tilde N-y_1|\le \frac C{\tilde N}.
\end{equation}
Note that $y_1$ depends only on $s$ and $x$ and so does $\tilde N$.
Further, for every $N>\tilde N$ one may find an $\alpha_1=\alpha_1(N)$
satisfying the inequalities (\ref{AlphaOneDistIneqDefEq})
and (\ref{AlphaOneIneqDefEq}) with $\tilde N$ replaced by $N$.

Applying the mean value theorem to  the line segment between
$\a_1/N$ and $y_1$,  it follows that
\begin{equation}
\label{SummandSecondEq}
e^{N(\langle x,\alpha_1/N-y_1\rangle+u_s(y_1)-u_s(\a_1/N))}
\geq
e^{-N|y_1-\a_1/N||\nabla u_s(y_2)|},
\end{equation}
where $y_2\in P\sm\del P$ is some point on the line segment between $\a_1/N$ and
$y_1$. 
Hence,
$$
\hbox{dist}(y_2,\partial P)>C/N.
$$
By Guillemin's
formula  (\ref{GuilleminFormulaEq}), we therefore have
$$
|\nabla u_s(y_2)|<C\log N+s||\dot u_0||_{C^1(P)}<C_T\log N, 
$$
for some constant $C_T$ that depends on $T$. It follows  that
\begin{equation}
\begin{array}{lll}
E_N(s,z) & \ge\dis \frac1N\log e^{N(\langle
x,\alpha_1(N)/N-y_1\rangle+u_s(y_1)-u_s(\a_1(N)/N))} \cr\cr\dis
&\ge\dis \frac1N\log e^{-C_T\log N}\ge \frac{-C_T\log N}N,
\end{array}
\end{equation}
and this concludes the proof of Lemma \ref{CZeroLemma}.
\qed

\bigskip

Lemma \ref{CZeroLemma} completes the proof of Theorem \ref{FirstMainThm}.

\bigskip

\bigskip\bigskip
\noindent {\bf Acknowledgments.}
This material is based upon work supported in part under National Science Foundation
grants DMS-0603850, 0904252.
Y.A.R. was also supported by graduate fellowships at M.I.T. and at Princeton University
during the academic year 2007-2008,
and by a National Science Foundation
Postdoctoral Research Fellowship at
Johns Hopkins University during the academic year 2008-2009.
Some of the results of this article were first presented in October 2008
at the conference ``Perspectives in Geometric Analysis" held at the BICMR.


\end{document}